\documentclass[A4,12pt,reqno]{amsart}
\usepackage{amsmath,amssymb,amsthm,mathrsfs,amsopn,amsfonts,amsbsy,amscd,bm}
\usepackage{bbm,pifont,color,graphicx,latexsym,fancyhdr,CJK}
\usepackage{multirow,bigdelim,layout,enumerate,tikz}
\usepackage[all]{xy}
\usepackage[flushmargin,hang]{footmisc}
\usepackage[top=.8in, bottom=.8in, left=.8in, right=.8in]{geometry}

\newtheorem{thm}{Theorem}[section]
\newtheorem{defn}[thm]{Definition}
\newtheorem{lem}[thm]{Lemma}
\newtheorem{prop}[thm]{Proposition}
\newtheorem{cor}[thm]{Corollary}
\numberwithin{equation}{thm}
\theoremstyle{definition}
\newtheorem{rmk}[thm]{Remark}

\newtheorem{exam}[thm]{Example}

\newtheorem{algo}[thm]{Algorithm}
\newcommand{\opname}[1]{\operatorname{\mathsf{#1}}}
  
\renewcommand{\ker}{\opname{Ker}}  \newcommand{\im}{\opname{Im}}       
    \newcommand{\Ext}{\opname{Ext}}

\renewcommand{\mod}{\opname{mod}}

\newcommand{\Hom}{\opname{Hom}}

\newcommand{\End}{\opname{End}}
\newcommand{\rad}{\opname{rad}}  \newcommand{\soc}{\opname{soc}}

\newcommand{\Rep}{\opname{Rep}}

\newcommand{\row}{\opname{row}}     \newcommand{\col}{\opname{col}}
\newcommand{\dg}{\opname{dg}}

\newcommand{\az}{{\alpha}}

  \newcommand{\ttz}{{\Theta}}
\newcommand{\dz}{{\delta}}  \newcommand{\ddz}{{\Delta}}
\newcommand{\gz}{{\gamma}}  
\newcommand{\sz}{{\sigma}}  \newcommand{\ssz}{{\Sigma}}
  
\newcommand{\lz}{{\lambda}}

\newcommand{\vz}{{\varphi}}

\newcommand{\cm}{{\mathcal M}}

\newcommand{\cz}{{\mathcal Z}}

 \newcommand{\fkC}{{\frak C}}
                             \newcommand{\fkD}{{\frak D}}
 
\newcommand{\fkg}{{\frak g}}
\newcommand{\fkh}{{\frak h}} \newcommand{\fkH}{{\frak H}}

\newcommand{\fkl}{{\frak l}}

\newcommand{\fks}{{\frak s}}

\newcommand{\bbd}{{\mathbb D}}
\newcommand{\bbf}{{\mathbb F}}
\newcommand{\bbn}{{\mathbb N}}
\newcommand{\bbq}{{\mathbb Q}}

\newcommand{\bbz}{{\mathbb Z}}

       \newcommand{\ol}[1]{\overline{#1}}
\newcommand{\ra}{\rightarrow}             
\newcommand{\lan}{{\langle}}              \newcommand{\ran}{{\rangle}}
\newcommand{\geqs}{{\geqslant}}           \newcommand{\leqs}{\leqslant}
\newcommand{\lra}{{\longrightarrow}}

\newcommand{\ttop}{\opname{top}}

\newcommand{\iso}{\stackrel{_\sim}{\rightarrow}}
\newcommand{\otm}{\otimes}
\newcommand{\bps}{\bigoplus}
\newcommand{\wit}{\widetilde}  \newcommand{\wih}{\widehat}

\newcommand{\llra}{~{\Longleftrightarrow}~}

\newcommand{\ie}{{\em i.e.}~}

\newcommand{\vartri}{\vartriangle}
\newcommand{\U}{{\bf U}}
\newcommand{\bdim}{{\bf dim}}
\newenvironment{psmallmatrix}{\left(\begin{smallmatrix}}{\end{smallmatrix}\right)}

\def\bsq{{\boldsymbol{q}}}

\def\th{{\text{\rm th}}}
\def\scc{{\textsc{c}}}
\def\fkh{{\mathfrak h}}

\begin{document}

\title[Canonical basis for quantum affine $\fkg\fkl_n$]%
{Multiplicaton formulas and canonical basis for quantum affine $\fkg\fkl_n$}

\author{Jie Du and Zhonghua Zhao$^\dag$}
\address{School of Mathematics and Statistics, University of New South Wales, Sydney 2052, Australia.}
\email{j.du@unsw.edu.au}
\address{Department of Mathematics and Computer Science, School of
Science, Beijing University of Chemical Technology, Beijing
100029, China.}
\email{zhaozh@mail.buct.edu.cn}


\date{\today}


\thanks{$^\dag$Corresponding author.\\
The research was carried out while the second author was visiting the University of New South Wales. The hospitality and support from UNSW are gratefully acknowledged. He also thanks the China Scholarship Council for the financial support.}

\maketitle

\begin{abstract}
We will give a representation-theoretic proof for the multiplication formula in the Ringel-Hall algebra $\fkH_\Delta(n)$ of a cyclic quiver $\Delta(n)$ given in \cite[Thm~4.5]{DuFu2015quantum}.  As a first application, we see immediately the existence of Hall polynomials for cyclic quivers,  a fact established in \cite{Guo1995hallpoly} and \cite{Ringel1993composition}, and derive a recursive formula to compute them. 
We will further use the formula and the construction of certain monomial base for $\fkH_\Delta(n)$ given in \cite{DengDuXiao2007generic}, together with the double Ringel--Hall algebra realisation of  the quantum loop algebra $\U_v(\wih{\fkg\fkl}_n)$ in \cite{DengDuFu2012double},  to develop some algorithms and to compute the canonical basis for $\U_v(\wih{\fkg\fkl}_n)^+$. As examples,
we will show explicitly the part of the canonical basis 
associated with modules of Lowey length at most $2$ for the quantum group $\U_v(\wih{\fkg\fkl}_2)$.
\end{abstract}

\setcounter{tocdepth}{1}

\section{Introduction}

The investigation on quantum algebras associated with affine Hecke algebras has made significant progress recently.
In the affine type $A$ case, an algebraic approach is developed in \cite{DengDuFu2012double} for the Schur--Weyl theory associated with the quantum loop algebra of $\mathfrak{gl}_n$, affine $q$-Schur algebras and Hecke algebras of the  affine symmetric groups. This approach, motivated from the algebraic approach for quantum $\mathfrak{gl}_n$, is different from the geometric approach developed in \cite{GV, Lu99}. Further in \cite{DuFu2015quantum, DF16}, new realisations for these quantum loop algebras and their integral Lusztig type form are obtained using affine $q$-Schur algebras. This generalises the work \cite{BeilinsonLusztigMacPherson1990geometric} of
Beilinson--Lusztig--MacPherson to this affine case.
For affine types of other than A, Fan et al used affine $q$-Schur algebras of type $C$ to construct in \cite{FLLLW} various types of quantum symmetric pairs. The multiplication formulas there are much more complicated, but can be used to study the modified versions of these quantum algebras and their canonical basis.  In this paper, we will see how a new multiplication formula discovered in \cite{DuFu2015quantum} is used to compute certain slices of the canonical basis for the $+$-part of the quantum loop algebra of $\mathfrak{gl}_n$.

The key ingredient of the approach developed in \cite{DengDuFu2012double} is the double Ringel-Hall algebra characterisation for the Drinfeld's quantum loop algebra of $\mathfrak{gl}_n$ \cite{Drinfeld1988new}. In this way, the Ringel--Hall algebra of a cyclic quiver and its opposite algebra become the $\pm$-part of the quantum loop algebra of $\mathfrak{gl}_n$, and their generators associated with the semisimple modules of the cyclic quiver play the role as done by usual Chevalley generators. In particular, the quantum affine Schur--Weyl duality can be described by explicit actions of these (infinitely many) generators associated with semisimple representations and a new realisation, i.e., a new construction of the quantum loop algebra of $\mathfrak{gl}_n$, is achieved through a beautiful multiplication formula of a basis element by a semisimple generator. It should be pointed out that these multiplication formulas are derived in the affine $q$-Schur algebras with most of the computation done within the affine Hecke algebras. However, when the formulas restrict to the $\pm$-part, they result in multiplication formulas for (generic) Ringel--Hall algebras of a cyclic quiver. Thus, a natural questions arises: Is there a direct proof for these formulas as a quantumization of Hall numbers associated with representations of a cyclic quiver over finite fields? 

In this paper, we first provide a representation-theoretic proof for the multiplication formula in the Ringel--Hall algebra (Theorem \ref{main-multip-theorem}). One key idea used in the proof is the bijective correspondence between the $m$-dimensional subspaces of an $n$-dimensional space and the reduced row echelon form of $m\times n$ matrices of rank $m$. We then use the multiplication formula to show in general the existence of Hall polynomials for cyclic quivers (c.f. \cite{Guo1995hallpoly} and \cite{Ringel1993composition}). As a further applications of the formula, we derive a recursively formulas for computing Hall polynomials and compute the canonical basis for (the $+$-part of) an quantum affine $\mathfrak{gl}_n$. This requires a systematic construction of a certain monomial basis. Thanks to \cite{DengDuXiao2007generic}, we will use the theory there to derive a couple of algorithms on matrices and will then follow them to produce the required monomial basis. 

Computing canonical bases is in general very difficult. Besides some lower rank cases of finite type (see, e.g., \cite[\S3]{Lusztig1990canonical} for types $A_1$ and $A_2$ and \cite{XicanonicalA31999, XicanonicalB21999} for type $A_3,B_2$) and certain tight monomials for quantum affine $\mathfrak{sl}_2$ (\cite{Lu93a}), there seems no explicit affine examples done in the literature. We now use the multiplication formua to compute several infinite series
of the canonical basis for $\U_v(\wih{\fkg\fkl}_2)$.
To ease the difficulty, we divided the basis into the so-called ``slices'' labelled by the Lowey length $\ell(M)$ and the periodicity $p(M)$ associated with a representation $M$ of a cyclic quiver. We explicitly compute several slices of the canonical basis associated with modules of Lowey length at most 2 for quantum affine $\mathfrak{gl}_2$. In a forthcoming paper, we will give further applications to the theory of quantum loop algebras of $\mathfrak{sl}_n$ developed in \cite{DengDuXiao2007generic}.
 
The paper is roughly divided into two parts. The first part from \S2 to \S4 deals with the theory of integral Hall 
algebras associated with finite fields, including the existence of Hall polynomials (Theorem \ref{Hall polynomials}) and a recursive formula (Corollary \ref{recursive formula}). The rest sections focus on computation of canonical basis for the (generic and twisted) Ringel--Hall algebras and quantum affine $\mathfrak{gl}_n$. With a selected monomial basis, we formulate Algorithm \ref{another const for hall alg} to compute the canonical basis. Five slices of the canonical basis for quantum affine $\mathfrak{gl}_2$ are explicitly worked out; see Propositions \ref{slice11} and \ref{slice20} and Theorems~\ref{theorem for case (2,1)} and \ref{slice22}.
\tableofcontents

\vspace{-3ex}
\subsection*{Notation}
For a positive integer $n$, let $M_{\vartri,n}(\bbz)$ be the set of all $\bbz\times\bbz$ matrices $A=(a_{i,j})_{i,j\in\bbz}$ with $a_{i,j}\in\bbz$
such that
\begin{itemize}
  \item [(1)] $a_{i,j}=a_{i+n,j+n}$ for $i,j\in\bbz$, and
  \item [(2)] for every $i\in\bbz$, both the set $\{j\in\bbz\mid a_{i,j}\neq0\}$ and $\{j\in\bbz\mid a_{j,i}\neq 0\}$ are finite.
\end{itemize}

Let $\ttz_\vartri(n)=M_{\vartri,n}(\bbn)$ be the subset of $M_{\vartri,n}(\bbz)$ consisting of matrices with entries from $\bbn$. Let
\begin{equation*}
  \ttz^+_\vartriangle(n)=\{A\in \ttz_\vartriangle(n)\mid a_{ij}=0~\text{for}~i\geqs j\}~\text{and}~\ttz^-_\vartriangle(n)=\{A\in \ttz_\vartriangle(n)\mid a_{ij}=0~\text{for}~i\leqs j\}.
\end{equation*}
For $A\in\ttz_\vartri(n)$, write
\begin{equation*}
  A=A^++A^0+A^-,
\end{equation*}
where $A^0$ is the diagonal submatrix of $A$, $A^+\in\ttz_\vartri^+(n),$ and $A^-\in\ttz_\vartri^-(n)$.

The {\it core} of  a matrix $A$ in $\Theta_\vartri^{+}(n)$ is the $n\times l$ submatrix of $A$ consisting of rows from 1 to $n$ and columns from 1 to $l$, where $l$ is the column index of the right most non-zero entry in the given $n$ rows.

Set $\bbz_{\vartri}^n=\{(\lz_i)_{i\in\bbz}\mid \lz_i\in\bbz,\lz_i=\lz_{i-n}~\text{for}~i\in\bbz\}$ and
$\bbn_{\vartri}^n=\{(\lz_i)_{i\in\bbz}\in\bbz_{\vartri}^n\mid \lz_i\geqs 0~\text{for}~i\in\bbz\}$.
For each $A\in M_{\vartri,n}(\bbz)$,
let
\begin{equation*}
  \row(A)=(\sum_{j\in\bbz}a_{i,j})_{i\in\bbz}\in \bbz_{\vartri}^n,\quad \col(A)=(\sum_{i\in\bbz}a_{i,j})_{j\in\bbz}\in \bbz_{\vartri}^n.
\end{equation*}

Define an order relation $\leqs$ on $\bbn_\vartri^n$ by 
$$\lz\leqs \mu\llra \lz_i\leqs \mu_i\;(1\leqs i\leqs n).$$
We say $\lz<\mu$ if $\lz\leqs\mu$ and $\lz\neq\mu$.

Let $\bbq(v)$ be the fraction field of $\cz:=\bbz[v,v^{-1}]$. For integers $N,t$ with $t\geqs 0$
and $\mu\in\bbz_{\vartri}^n$ and $\lz\in\bbn_{\vartri}^n$,
define Gaussian polynomial and their symmetric version in $\cz$:
\begin{equation*}
  \left[\!\!\left[N\atop t\right]\!\!\right]=\dfrac{[\![\begin{matrix} N \end{matrix}]\!]!}
  {[\![\begin{matrix} t \end{matrix}]\!]![\![\begin{matrix} N-t \end{matrix}]\!]!}=\prod_{1\leqs i\leqs t}\dfrac{v^{2(N-i+1)-1}}{v^{2i}-1}\quad
   \quad\text{and}\quad
   \begin{bmatrix}
     N\\t
   \end{bmatrix}=v^{-t(N-t)}\left[\!\!\left[\begin{matrix} N\\ t \end{matrix}\right]\!\!\right],
\end{equation*}
where $
  [\![\begin{matrix} t \end{matrix}]\!]!=[\![\begin{matrix} 1 \end{matrix}]\!][\![\begin{matrix} 2 \end{matrix}]\!]\cdots
  [\![\begin{matrix} t \end{matrix}]\!]\quad\text{with}\quad [\![\begin{matrix} m \end{matrix}]\!]=\dfrac{v^{2m}-1}{v^2-1}.
$

For a prime power $q$, we write $\left[\!\!\left[\begin{matrix} N\\ t \end{matrix}\right]\!\!\right]_q$ for the value of the polynomial at $v^2=q$.
\section{The integral Hall algebras of cyclic quivers and Hall polynomials}
Let $\Delta=\Delta(n)$ $(n\geqs 2)$ be the cyclic quiver
with vertex set $I:=\bbz/n\bbz=\{1,2,\cdots,n\}$ and arrow set $\{i\ra i+1\mid i\in I\}$, and let $k\Delta$ be the path
algebra of $\Delta$ over a field $k$. For a representation $M=(V_i,f_i)_i$ of $\Delta$,
let ${\bf dim}M=(\dim V_1,\dim V_2,\cdots,\dim V_n)\in\bbn I=\bbn^n$ and $\dim M=\sum\limits_{i=1}^n\dim V_i$
denote the dimension vector and the dimension of $M$, respectively, and let $[M]$ denote the isoclass (isomorphism class) of $M$.

A representation $M=(V_i,f_i)_i$ of $\Delta$ over $k$ (or a $k\Delta$-module) is called {\em nilpotent} if the composition $f_n\cdots f_2f_1:V_1\ra V_1$
is nilpotent, or equivalently, one of the $f_{i-1}\cdots f_nf_1\cdots f_i:V_i\ra V_i$ $(2\leqs i\leqs n)$ is nilpotent. By $\Rep^0\Delta=\Rep^0_k\Delta(n)$
we denote the category of finite dimensional nilpotent representations of $\Delta(n)$ over $k$. For each vertex $i\in I$, there is a
one-dimensional representation $S_i$ in $\Rep^0\Delta$ satisfying $(S_i)_i=k$ and $(S_i)_j=0$ for $j\neq i$. It is known that $\{S_i\mid i\in I\}$ forms
a complete set of simple objects in $\Rep^0\Delta$. 

For $M\in\Rep^0\Delta$, we denote by $\rad(M)$ the radical of $M$, \ie the intersection of all maximal submodules of $M$,
and by $\ttop(M)=M/\rad(M)$, the top of $M$.

Up to isomorphism, all non-isomorphic indecomposable representations in $\Rep^0\Delta$ are given by $S_i[l]$ $(i\in I~\text{and}~l\geqs 1)$ of length $l$ with top $S_i$. Note that $S_i[l]$ can be described by vector spaces and linear maps around the cyclic quiver:
\begin{equation}\label{ind mod}
\xymatrix{
0\ar[r]^(0.53){0}&k\ar[r]^(0.53){1}&k
\ar[r]^(0.53){1}&\cdots\ar[r]^(0.53){1}&k\ar[r]^(0.53){1}&k\ar[r]^(0.53){0}&0\cdots
}\end{equation}
Here the number of $k$'s is $l$ and the first $k$ is at vertex $i$, the second at $i+1$, ..., the $(n+i)$th is again at vertex $i=n+i$, etc.

For $i<j$, set
\begin{equation*}
  M^{i,j}=S_i[j-i]~\text{and}~M^{i+n,j+n}=M^{i,j}.
\end{equation*}
For any $A=(a_{i,j})\in\ttz_\vartri^+(n)$, let
\begin{equation*}
  M(A)=M_k(A)=\bps_{1\leqs i\leqs n,i<j}a_{i,j}M^{i,j}.
\end{equation*}
Then the set $\{M_k(A)\mid A\in \ttz_{\vartri}^+(n)$ forms a complete set of all non-isomorphic finite dimensional nilpotent representations of $\ddz(n)$. If $k$ is a finite field of $q=q_k$ elements, we write
$M_q(A)=M_k(A)$.

Every element $\alpha=(\alpha_i)_{i\in\bbz}\in\bbn_\vartri^n$ defines a semisimple representation
$$S_\alpha=\oplus_{i=1}^n\alpha_iS_i.$$

A matrix $A=(a_{i,j})\in\ttz_{\vartri}^+(n)$ is called {\em aperiodic} if, for each $l\geqs 1$, there exists $i\in\bbz$ such that $a_{i,i+l}=0$.
Otherwise, $A$ is called {\em periodic}. A nilpotent representation $M(A)$ is called aperiodic (resp. periodic) if $A$ is aperiodic (resp. periodic).


Associated to a cyclic quiver, Ringel introduced an associative algebra, the {\em Hall algebra}, which can be defined at two levels: the integral level and the generic level.

For $A,B,C\in \ttz_{\vartri}^+(n)$ and any prime power $q$, let $\fkh_{M_q(B),M_q(C)}^{M_q(A)}$ be the number of submodules $N$ of $M_q(A)$ such that
$N\cong M_q(C)$ and $M_q(A)/N\cong M_q(B)$. More generally, given $A,B_1,B_2,\cdots,B_m\in\ttz_{\vartri}^+(n)$, denote by $\fkh_{M_q(B_1),M_q(B_2),\cdots,M_q(B_m)}^{M_q(A)}$
the number of filtrations
\begin{equation*}
  M_q(A)=M_0\supseteq M_1\supseteq M_2\supseteq \cdots\supseteq M_{m-1}\subseteq M_m=0,
\end{equation*}
such that $M_{t-1}/M_t\cong M_q(B_t)$ for $1\leqs t\leqs m$. 

The {\it (integral) Hall algebra} $\fkH_\vartri^\diamond(n,q)$ associated with $\Rep^0_k\ddz(n)$ over a finite field $k$ of $q$ elements, is the free $\bbz$-module spanned by basis $\{u_{A,q}:=u_{[M_q(A)]}\mid A\in\}$ with multiplication\footnote{The multiplication is denoted by $\circ$ in \cite{DengDuXiao2007generic}.} given by .
\begin{equation*}
  u_{B,q}\diamond u_{C,q}=\sum_{A\in\ttz_{\vartri}^+(n)}\fkh_{M_q(B),M_q(C)}^{M_q(A)}u_{A,q}.
\end{equation*}

By a result in \cite{Guo1995hallpoly,Ringel1993composition}, the Hall numbers $\fkh_{M_q(B),M_q(C)}^{M_q(A)}$ are polynomials in $q$ with integral coefficients. We now provide an independent proof for the fact, building on the following multiplication formula. A generic version of this formula is given by Fu and the first author in \cite{DuFu2015quantum},
using the technique of Hecke algebras, affine $q$-Schur algebras, and the new realisation of the quantum loop algebra of $\mathfrak{gl}_n$.

\begin{thm}\label{main-multip-theorem}
For $A\in\Theta_{\vartriangle}^+(n),~\az=(\az_i)_{i\in\bbz}\in\bbn_{\vartriangle}^n$, we have the following multiplication formula in the Hall algebra $\fkH_\vartri^\diamond(n,q)$:
  \begin{equation*}
    u_{\az,q}\diamond u_{A,q}=\sum_{\stackrel{T\in \Theta_{\vartriangle}^+(n)}{\row(T)=\az}}{q}^{\sum\limits_{\stackrel{1\leqs i\leqs n}{i<l<j}}(a_{ij}t_{il}-t_{ij}t_{i+1,l})}\prod_{\stackrel{1\leqs i\leqs n}{j\in\bbz,j>i}}\left[\!\!\left[\begin{matrix} a_{ij}+t_{ij}-t_{i-1,j}\\ t_{ij} \end{matrix}\right]\!\!\right]_qu_{A+T-\widetilde{T}^+,q},
  \end{equation*}
where
$  ^\thicksim:\ttz_\vartri(n)\ra\ttz_\vartri(n), A=(a_{i,j})\mapsto \wit{A}=(\wit{a}_{i,j})$
is the row-descending map defined by $\wit{a}_{i,j}=a_{i-1,j}$ for all $i,j\in\bbz$ and $\wit{T}^+$ denotes the upper triangular submatrix of $\wit T$. 
\end{thm}
 
We will prove this result in the next section. We first use the formula to prove the existence of Hall polynomials.

Let $\cm$ be the set of all isoclasses of representation in $\Rep^0\ddz(n)$.
Given two objects $M,N\in\Rep^0\ddz(n)$, there exists a unique (up to isomorphism) extension $G$ of $M$ by $N$ with minimum $\dim\End(G)$\cite{Bongartz1996degenerations,Reineke2001generic,DengDu2005monomial,DengDuParashallWang2008finite}.
The extension $G$ is called the {\em generic extension}\footnote{There exists geometrical description
when the field $k$ is algebraically closed, for details, see \cite{Reineke2001generic}.} of $M$ by $N$ and is
denoted by $G=M*N$. If we define $[M]*[N]=[M*N]$, then it is known from \cite{Reineke2001generic} that $*$ is
associative and $(\cm,*)$ is a monoid with identity $[0]$.

Besides the monoid structure, $\cm$ has also a poset structure. For two nilpotent representations $M,N\in\Rep^0\ddz(n)$ with
${\bf dim} M={\bf dim} N$, define
$$N\leqs_{\dg}M\llra \dim\Hom(X,N)\geqs\dim\Hom(X,M),~\text{for all}~X\in\Rep^0\ddz(n).$$
see \cite{Zwara1997degenerations}. This gives rise to a partial order on the set of isoclasses of representations in $\Rep^0\ddz(n)$,
called the {\em degeneration order}. Thus, it also induces a partial order on $\ttz_\vartri^+(n)$ by setting
$$A\leqs_{\dg} B\llra M(A)\leqs_{\dg} M(B).$$

Following \cite{BeilinsonLusztigMacPherson1990geometric} and \cite{DuFu2010modified} we may define the order relation
$\preccurlyeq$ on $M_{\vartri,n}(\bbz)$ as follows. For $A\in M_{\vartri,n}(\bbz)$ and $i\neq j\in\bbz$, let
\begin{equation*}\sz_{i,j}(A)=
  \begin{cases}
    \sum\limits_{s\leqs i,t\geqs j}a_{s,t},&\text{if}~i<j,\\
    \sum\limits_{s\geqs i,t\leqs j}a_{s,t},&\text{if}~i>j.
  \end{cases}
\end{equation*}
For $A,B\in M_{\vartri,n}(\bbz)$, define
$$B\preccurlyeq A~\text{if and only if~}\sz_{i,j}(B)\leqs\sz_{i,j}(A)~\text{for all}~i\neq j.$$
Set $B\prec A$ if $B\preccurlyeq A$, and for some $(i,j)$ with $i\neq j,\sz_{i,j}(B)<\sz_{i,j}(A)$. 

Note that restricting the order relation to $\ttz_{\vartri}^+(n)$ gives a poset $(\ttz_{\vartri}^+(n),\preccurlyeq)$.
Note also from \cite[Theorem 6.2]{DuFu2010modified} that, if $A,B\in\ttz_{\vartri}^+(n)$,
then 
\begin{equation}\label{dg order}
B\leqs_{\dg}A\llra B\preccurlyeq A \text{ and }  \bdim M(A)=\bdim M(B).
\end{equation}
Thus, $(\ttz_{\vartri}^+(n),\leqs_{\dg})$ is also a poset.

An element $\lz\in\bbn_\vartri^n$ is called {\em sincere} if $\lz_i>0$ for all $i\in I$. Let 
$$I^{\text{sin}}=\{\text{all sincere vectors in }\bbn_\vartri^n\}\;\;\text{ and }\;\;
\wit{I}=I\cup I^{\text{sin}}.$$ 
For $X\in\{I,I^{\text{sin}},\wit I\}$,  Let $\ssz_X$ be the set of words on the alphabet $X$ and let $\wit{\ssz}=\ssz_{\wit I}$.

For each $w={\bm a}_1{\bm a}_2\cdots{\bm a}_m\in\wit{\ssz}$, we set $M(w)=S_{{\bm a}_1}*S_{{\bm a}_2}*\cdots*S_{{\bm a}_m}$.
Then there is a unique $A\in\ttz_\vartri^+(n)$ such that $M(w)\cong M(A)$, and we set $\wp(w)=A$, which induces a surjective map $\wp:\wit{\ssz}\lra\ttz_\vartri^+(n),w\mapsto \wp(w)$. Note that $\wp$ induces a surjective map $\wp:\ssz\ra\ttz_\vartri^{ap}(n)$.

For $\bm a\in\wit{I}$, set $u_{\bm a}=u_{[S_{\bm a}]}$. For any $w={\bm a}_1{\bm a}_2\cdots{\bm a}_m\in\wit{\ssz}$ and $A\in \ttz_\vartri^{+}(n)$, repeatedly applying Theorem \ref{main-multip-theorem} shows that there exists a polynomial $\vz^A_w(\bm q)\in\bbz[\bm q_k]$ such that $\vz^A_w(q)=\fkh^M_{M_1,M_2,\cdots,M_m}$ with $M_i\cong S_{\bm a_i}$ and $M\cong M_k(A)$. 

Any word $w={\bm a}_1{\bm a}_2\cdots{\bm a}_m\in\wit{\ssz}$ can be uniquely expressed in the {\em tight form} $w={\bm b}_1^{e_1}{\bm b}_2^{e_2}\cdots{\bm b}_t^{e_t}$ where $e_i=1$ if ${\bm b_i}$ is sincere, and $e_i$ is the number of consecutive occurrence of $\bm b_i$ if $\bm b_i\in I$. By \cite[Lem.~5.1]{DengDuXiao2007generic} (see also the proof of \cite[Prop.~9.1]{DengDu2005monomial}), $\varphi_w^{A}$ is divisible by $\prod_{i=1}^t[\![e_i]\!]^!$ for every $A\preceq \wp(w)$. Thus, there exists $\gamma_w^{A}\in\bbz[\bm q]$ such that 
$$\varphi_w^{A}=\prod_{i=1}^t[\![e_i]\!]^!\gamma_w^{A}\in\bbz[\bm q].$$
Note that the polynomials $\gamma_w^{A}$ are also Hall polynomials.
In fact,
for a finite field $k$, we have $\gamma^A_w(q_k)=\fkh^M_{N_1,N_2,\cdots,N_m}$ with $N_i\cong e_iS_{\bm b_i}$ and $M\cong M_k(A)$. A word $w$ is called {\em distinguished} if the Hall polynomial $\gz^{\wp(w)}_w(\bsq)=1$. 


As a first application, we now use the multiplication formula to prove the existence of Hall polynomials. This result was first given in \cite{Guo1995hallpoly}, \cite[8.1]{Ringel1993composition}. 
\begin{thm}\label{Hall polynomials}
 The Hall numbers $\fkh_{M_q(B),M_q(C)}^{M_q(A)}$ associated with  $A,B,C\in\ttz_{\vartri}^+(n)$ and any prime power $q$ are polynomials in $q$. In other words, there exist $\varphi_{B,C}^A\in\bbz[\bm q]$ such that
 $\varphi_{B,C}^A(q)=\fkh_{M_q(B),M_q(C)}^{M_q(A)}$ for all such $q$.
\end{thm}

\begin{proof} For $w={\bm b}_1{\bm b}_2\cdots{\bm b}_t\in\wit{\ssz}$, if
we write in $\fkH_\vartri^\diamond(n,q)$
$$u_{w,q}=u_{{{\bm b}_1,q}}\diamond\cdots \diamond u_{{{\bm b}_m,q}}=\sum_{B'\preceq \wp(w)}\fkh_{w}^{B'}u_{B',q},$$
Then, by Theorem \ref{main-multip-theorem}, there exist polynomials $\varphi_w^{B'}$ such that $\varphi_w^{B'}(q_k)=\fkh_{w}^{B'}$.  Assume now $w$ is distinguished (see \cite[Th.~6.2]{DengDuXiao2007generic}) such that $B=\wp(w)$, $M=M(B)$, $L=M(A)$ and $N=M(C)$. Then $\varphi_w^B=\prod_{i=1}^t[\![e_i]\!]^!$ and $\varphi_w^{B'}/\varphi_w^B=\gamma_w^{B'}$ are all polynomials.

Now, by Theorem \ref{main-multip-theorem} again,
 the Hall numbers in $u_w\diamond u_C=\sum_{A'\preceq B*C}\fkh_{w,C}^{A}u_{A}$ are the values of certain polynomials $\varphi_{w,C}^{A}$ at $q_k$.
 On the other hand,
$$\aligned
u_w\diamond u_C&=\sum_{B'\preceq B}\fkh_{w}^{B'}(u_{B'}\diamond u_C)\\
&=\fkh_{w}^Bu_B\diamond u_C+\sum_{B'\prec B}\fkh_{w}^{B'}(u_{B'}\diamond u_C)\\
&=\fkh_{w}^Bu_B\diamond u_C+\sum_{A\prec B'*C}\bigg(\sum_{B'\prec B}\fkh_{w}^{B'}\fkh_{B',C}^{A}\bigg)u_{A}.
\endaligned$$
By equating coefficients, we see that all polynomials  $\varphi_{w,C}^{A}$ is divisible by $\varphi_w^B$.
Thus, we have
$$\fkh_{w}^Bu_B\diamond u_C=\sum_{A\preceq B*C}\fkh_{w,C}^{A}u_{A}-\sum_{A\prec B'*C}\bigg(\sum_{B'\prec B}\fkh_{w}^{B'}\fkh_{B',C}^{A}\bigg)u_{A}$$
Now the assertion follows from induction on $\preceq$.
\end{proof}
In \S4, we will give algorithms to compute distinguished words $w_A$ associated with each $A\in\ttz_\vartri^{+}(n)$ and to derive a recursive formula for Hall polynomials.

\section{Proof of Theorem \ref{main-multip-theorem}}

Recall that a matrix over a field in row-echelon form is said to be in {\it reduced row-echelon form} (RREF) if every leading column has 1 at the leading entry and 0 elsewhere.

\begin{lem}\label{countfinitefield}Let $\mathcal R_{m,n}\subseteq M_{m,n}(\bbf_q)$ be the subset consisting of all $m\times n$ matrices in reduced row-echelon form and of rank $m$. Then
$$|\mathcal R_{m,n}|=\left[\!\!\left[n\atop m\right]\!\!\right]_{v^2=q}.$$

\end{lem}

\begin{proof} Let $\mathcal V_{m,n} $ be the set of all dimension $m$ subspaces of $\bbf_q^n$. Then,
for $T\in \mathcal R_{m,n}$, the rows of $T$ spans a subspace $V_T$ of dimension $m$. Thus, we have a map
$$f:\mathcal R_{m,n}\longrightarrow\mathcal V_{m,n}, T\longmapsto V_T.$$
Clearly, $f$ is surjective. It is not hard to see that $f$ is also injective. Now, the assertion follows from the bijection.
\end{proof}

\begin{prop}\label{basiccountinglem}
  For $i\in I$, $a_t,d_t,m\in\bbz$ with $a_t\geqs d_t\geqs0,~m\geqs1$, $t=1,2,\cdots,m$, and representations 
  \begin{equation*}
  \begin{split}
    L&=a_1S_i\oplus a_2S_i[2]\oplus\cdots\oplus a_mS_i[m],\quad M=(d_1+\cdots+d_m)S_i, \text{ and }\\
    N&=(a_1-d_1)S_i\oplus\big((a_2-d_2)S_i[2]\oplus d_2S_{i+1}\big)\oplus\cdots\oplus\big((a_m-d_m)S_i[m]\oplus d_mS_{i+1}[m-1]\big),
  \end{split}
  \end{equation*}
in {\rm Rep}$^0_k(\Delta)$, 
 the Hall number $\fkh^L_{M,N}$ is a polynomial in $q=q_k$:
  $$\fkh^L_{M,N}=q^{\sum\limits_{1\leqs k<l\leqs m}d_k(a_l-d_l)}\left[\!\!\left[\begin{matrix} a_1\\ d_1 \end{matrix}\right]\!\!\right]_q\left[\!\!\left[\begin{matrix} a_2\\ d_2 \end{matrix}\right]\!\!\right]_q\cdots
  \left[\!\!\left[\begin{matrix} a_m\\ d_m \end{matrix}\right]\!\!\right]_q.$$
\end{prop}

\begin{proof}
Without loss, we may assume $i=1$.
Represent the modules $L,N$ by vector spaces and linear maps around the cyclic quiver as follows (cf. \eqref{ind mod}):
$$\xymatrix{
L:k^{a_1+a_2+\cdots+a_m}\ar[r]^(0.54){p_1}&k^{a_2+a_3+\cdots+a_m}\ar[r]^(0.53){p_2}&k^{a_3+\cdots+a_m}
\ar[r]^(0.6){p_3}&\cdots\ar[r]^(0.35){p_{m-2}}&k^{a_{m-1}+a_m}\ar[r]^(0.6){p_{m-1}}&k^{a_m}\\
N:k^{a_1-d_1+a_2-d_2+\cdots+a_m-d_m}\ar[r]^(0.64){f}&k^{a_2+a_3+\cdots+a_m}\ar[r]^(0.53){p_2}&k^{a_3+\cdots+a_m}
\ar[r]^(0.6){p_3}&\cdots\ar[r]^(0.35){p_{m-2}}&k^{a_{m-1}+a_m}\ar[r]^(0.6){p_{m-1}}&k^{a_m}.
}$$
Here $p_i$ is the projection map defined by the matrix $[0_{a_i}~I_{\widetilde a_{i+1}}]$, where
$$\widetilde a_i:=a_{i}+\cdots+a_m$$  and  $0_{a_i}$
 is the $\widetilde a_{i+1}\times a_i$ zero matrix, while $f$ is the restriction of $p_1$. Thus, $f$ projects the component $k^{a_1-d_1}$ to 0 and imbeds the component $k^{a_i-d_i}$ for $i\geq2$ into the component $k^{a_i}$ via the $a_i\times (d_i-a_i)$ matrix $J_i=\bigg(\begin{matrix}I_{d_i-a_i}\\ 0\end{matrix}\bigg)$. In other words, $f$ is defined by the $\widetilde a_2\times (\widetilde a_1-\widetilde d_1)$ matrix $A$ with blocks $J_1, J_2,\ldots,J_m$ on the diagonal, where $J_1$ is the $a_1\times (d_1-a_1)$ zero matrix.

Let $U\leq L$ be a submodule such that $U\cong N,L/U\cong M$. Then $U=\ker(g)$ for some module epimorphism $g:L\to M$. Thus, the short exact sequence $0\to U\to L\to M\to0$ gives the following commutative diagram:
$$\xymatrix{
U\ar[d]^\iota&\ker g_1\ar[d]^{\iota_1}\ar[r]^(0.4){p_1}&k^{a_2+a_3+\cdots+a_m}\ar[d]^{id}\ar[r]^(0.53){p_2}&k^{a_3+\cdots+a_m}\ar[d]^{id}
\ar[r]^(0.6){p_3}&\cdots\ar[r]^(0.35){p_{m-2}}&k^{a_{m-1}+a_m}\ar[d]^{id}\ar[r]^(0.6){p_{m-1}}&k^{a_m}\ar[d]^{id}\\
L\ar[d]^g&k^{a_1+a_2+\cdots+a_m}\ar[d]^{g_1}\ar[r]^(0.5){p_1}&k^{a_2+a_3+\cdots+a_m}\ar[d]\ar[r]^(0.53){p_2}&k^{a_3+\cdots+a_m}\ar[d]\ar[r]^(0.6){p_3}
&\cdots\ar[r]^(0.35){p_{m-2}}&k^{a_{m-1}+a_m}\ar[d]\ar[r]^(0.6){p_{m-1}}&k^{a_m}\ar[d]\\
M\ar[d]&k^{d_1+d_2+\cdots+d_m}\ar[d]\ar[r]&0\ar[r]&0\ar[r]&\cdots\ar[r]&0\ar[r]&0\\
0& 0
}$$
Since $g$ is surjective, it is easy to see $\ker g_1\cong k^{a_1-d_1+\cdots+a_m-d_m}$ as vector spaces.
Represent the linear map $g_1:k^{a_1+\cdots+a_m}\ra k^{d_1+\cdots+d_m}$ by a $\widetilde d_1\times\widetilde a_1$ matrix $T_U$ in reduced row-echelon form. Since $g_1$ is onto, $T_U$ is an upper triangular matrix with $\widetilde d_1$ leading columns and $\ell=\widetilde a_1-\widetilde d_1$ non-leading columns, corresponding to $\ell$ free variables $x_{i_1},x_{i_2},\ldots,x_{i_\ell}$. Let $v_j$ be the solution to $T_Ux=0$ obtained by setting $x_{i_j}=1$ and other free variables to 0.
Then, $\ker g_1$ has a basis $v_1,v_2,\ldots,v_\ell$.

Since $U\cong N$, there exist linear isomorphism $\phi=(\phi_1,\phi_2,\cdots,\phi_m)$ making the following diagram commutes
$$\xymatrix{
U\ar[d]_{\cong}^\phi&\ker g_1\ar[d]^{\phi_1}\ar[r]^(0.45){p_1}&k^{a_2+a_3+\cdots+a_m}\ar[d]^{\phi_2}\ar[r]^(0.53){p_2}&k^{a_3+\cdots+a_m}\ar[d]^{\phi_3}
\ar[r]^(0.6){p_3}&\cdots\ar[r]^(0.35){p_{m-2}}&k^{a_{m-1}+a_m}\ar[d]^{\phi_{m-1}}\ar[r]^(0.6){p_{m-1}}&k^{a_m}\ar[d]^{\phi_m}\\
N&\!\!\!\!\!k^{a_1-d_1+a_2-d_2+\cdots+a_m-d_m}\ar[r]^(0.6){f}&k^{a_2+a_3+\cdots+a_m}\ar[r]^(0.53){p_2}&k^{a_3+\cdots+a_m}\ar[r]^(0.6){p_3}
&\cdots\ar[r]^(0.35){p_{m-2}}&k^{a_{m-1}+a_m}\ar[r]^(0.6){p_{m-1}}&k^{a_m}
}$$
Hence, the images of $p_i\ldots p_2p_1$ in the top row maps must have the same dimension as that of the map $p_i\ldots p_2f$ below.
Since the dimension of $\im(f)$ is $\widetilde a_2-\widetilde d_2$, $p_1$ must send $v_1,\ldots,v_{a_1-d_1}$ to 0. This forces the first $a_1$ columns contains $d_1$ leading columns. Similarly, $\dim\im(p_2p_1)=\dim\im(p_2f)$ forces the next $a_2$ columns in $T_U$ contains $d_2$ leading columns, and so on. This proves that, if $T_U$ is divided in $d_i\times a_j$ blocks, then $T_U$ is upper triangular with $m$ $(d_i\times a_i)$-blocks on the diagonal each of which has rank $d_i$.

Let
$\mathcal T$ be the subset of all $T\in M_{\widetilde d_1,\widetilde a_1}(\bbf_q)$ such that $T$ is in RREF and $T$ has $m$ $(d_i\times a_i)$-blocks $B_i$ on the diagonal each of which has rank $d_i$. The arguement above shows that the map $U\mapsto T_U$ is a bijection from the set $\{U\subseteq L\mid U\cong N,L/U\cong M\}$ to $\mathcal T$. Hence, $\fkh^L_{M,N}=|\mathcal T|$.

Now, to form such a matrix $T$, by Lemma \ref{countfinitefield}, the number of the $(d_1\times a_1)$-block $B_1$ is $\big[\!\!\big[{a_1\atop d_1}\big]\!\!\big]$ and the number of other $(d_1\times a_i)$-blocks for $i\geq2$ in the first $d_1$ rows is $q^{d_1(a_2-d_2+a_3-d_3+\cdots+a_m-d_m)}$. Counting the number of the blocks in the next $d_2$ rows, $d_3$ rows, $\ldots$, similarly, yields
$$|\mathcal T|=q^{d_1(a_2-d_2+a_3-d_3+\cdots+a_m-d_m)}\left[\!\!\left[a_1\atop d_1\right]\!\!\right]_q\times q^{d_2(a_3-d_3+\cdots+a_m-d_m)}\left[\!\!\left[a_2\atop d_2\right]\!\!\right]_q\times\cdots\times q^{d_{m-1}(a_m-d_m)} \left[\!\!\left[a_m\atop d_m\right]\!\!\right]_q,$$
as desired.
\end{proof}

\begin{rmk} A dual version of the above result, where the roles $M$ and $N$ are swapped, is known in \cite[\S2.2]{Sc06} and have been used in \cite[Lem.~2.3.5]{FL}. Unlike the representation-theoretic proof above, the proof in loc. cit. involves the geometry of the
Grassmanian variety.
\end{rmk}

\begin{lem}\label{product lem} For nilpotent representations $L,M,N$ of $\triangle(n)$, if $N\leq L$ and $L/N\cong M$ is semisimple, then there exists submodules $L_i\leq L$  $N_i\leq N$ and $M_i\leq M$ such that $L=\bigoplus_{i=1}^nL_i$, $N=\bigoplus_{i=1}^nN_i$, $M=\bigoplus_{i=1}^nM_i$ and
$$\fkh_{M,N}^L=\prod_{i=1}^n\fkh_{M_i,N_i}^{L_i}.$$
\end{lem}
\begin{proof} Let top$(L)_i$ denote the isotypic component of top$(L)$ associated with $S_i$. Then $L=\oplus_{i=1}^nL_i$ where top$(L_i)=\text{top}(M)_i$. Thus, if $M_i$ denotes the isotypic component of $M$ associated with $S_i$ and $\pi:L\to M$ denotes the quotient map, then restriction defines an epimorphism $\pi_i=\pi|_{L_i}:L_i\to M_i$. Let $N_i=\pi_i^{-1}(M_i)$. Then $N_i=L_i\cap N$ and $N=\oplus_{i=1}^nN_i$.
Now, our assertion follows from the following bijection
$$\aligned
\prod_{i+1}^n\{U_i\leq L_i\mid U_i\cong N_i,L_i/U_i\cong M_i\}&\longrightarrow \{U\leq L\mid U\cong N,L/U\cong M\},\\
(U_1,\ldots,U_n)&\longmapsto U_1+\ldots+ U_n,\endaligned$$
noting $U=(U\cap L_1)+\cdots+(U\cap L_n)$.
\end{proof}

We are now ready to give a representation-theoretic proof for the multiplication formula in \cite[Th. 4.5]{DuFu2015quantum}.
As mentioned in the introduction, this formula is the restriction to the positive part of certain multiplication formulas for the quantum loop algebra of $\mathfrak{gl}_n$ \cite[Prop.~4.2]{DuFu2015quantum}, which is obtained from lifting some multiplication formulas in the affine $q$-Schur algebras associated with the affine Hecke algebra. See \cite[Prop.~2.3.6]{FL} for a geometric proof building on the Hall polynomials computed in  \cite[\S2.2]{Sc06}.

\begin{proof}[Proof of Theorem \ref{main-multip-theorem}]
We first claim that, if $L$ is an extension of the semisimple representation $S_\alpha$ by $N=M(A)$, then $L\cong M(A+T-\wit{T}^+)$ for some $T\in \ttz_\vartri^+(n)$ with $\alpha=\row(T)$. Indeed, suppose
$L\cong M(C)$ for some $C=(c_{i,j})$ and decompose $L=\oplus_{i=1}^nL_i$ as in Lemma~\ref{product lem}. If $U\leq L$ is a submodule isomorphic to $N$, then there exist $t_{ij}\in \mathbb N$ such that
$U_i=U\cap L_i\cong\bigoplus_{i<j}\big((c_{ij}-t_{ij})S_i[j-i]\oplus t_{ij}S_{i+1}[j-i-1]\big)$, where $\sum\limits_{i<j}t_{ij}=\az_i$.
Thus, $U\cong N$ becomes
$$\bigoplus_{i=1}^n\bigoplus_{i<j}\big((c_{ij}-t_{ij})S_i[j-i]\oplus t_{ij}S_{i+1}[j-i-1]\big)
\cong \bigoplus_{i=1}^n\bigoplus_{i<j}a_{ij}S_i[j-i].$$
By the Krull--Remak--Schmidt theorem, we have
\begin{equation}\label{cor1}
  c_{ij}-t_{ij}+t_{i-1,j}=a_{ij}\quad\text{for all } i<j \text{ with }i=1,2,\cdots,n.
\end{equation}
Hence, if we form the upper triangular matrix $T=(t_{i,j})\in\Theta_{\vartriangle}^+(n)$ then $C=A+T-\widetilde T^+$, proving the claim.

For $C=A+T-\wit{T}^+$, by Lemma  \ref{product lem}, we have
\begin{equation}\label{Hall1}
\fkH^{C}_{S_\az,A}=\prod_{i=1}^n\fkH^{L_i}_{M_i,N_i},
\end{equation}
where $L_i\cong\bigoplus_{{j>i}}(a_{ij}+t_{ij}-t_{i-1,j})S_i[j-i]$, $M_i\cong\bigoplus_{j>i}t_{ij}S_i$ and $N_i\cong(a_{ij}-t_{i-1,j})S_i[j-i]\oplus t_{ij}S_{i+1}[j-i-1]$.
Applying Proposition \ref{basiccountinglem} with $a_l=a_{i,i+l}-t_{i,i+l}-t_{i-1,i+l}$, $d_l=t_{i,i+l}$ yields
\begin{equation}\label{Hall2}
\fkH^{L_i}_{M_i,N_i}=q^{\sum_{l,j\in\mathbb Z\atop i<l<j}t_{il}(a_{ij}-t_{i-1,j})}
\prod_{j\in\mathbb Z,i<j}\bigg[\!\!\bigg[\begin{matrix} a_{ij}+t_{ij}-t_{i-1,j}\\ t_{ij} \end{matrix}\bigg]\!\!\bigg]_q\qquad(q=q_k).
\end{equation}
Finally, it remains to prove
\begin{equation}\label{cor2}
\sum_{i=1}^n\sum_{i<l<j}t_{il}(a_{ij}-t_{i-1,j})=\sum_{i=1}^n\sum_{i<l<j}(a_{ij}t_{il}-t_{ij}t_{i+1,l}),
\end{equation}
or, equivalently, to prove
\begin{equation*}
\sum_{\stackrel{1\leqs i\leqs n}{i<l<j}}t_{il}t_{i-1,j}=\sum_{\stackrel{1\leqs i\leqs n}{i<l<j}}t_{ij}t_{i+1,l}.
\end{equation*}
 This follows from the fact that the sets $J_1=\{t_{il}t_{i-1,j}\neq0\mid 1\leqs i\leqs n,i<l<j\}$ and $J_2=\{t_{ij}t_{i+1,l}\neq0\mid 1\leqs i\leqs n,i<l<j\}$ are identical. To see this, take $t_{il}t_{i-1,j}\in J_1$ where $i<l<j$. If $2\leqs i\leqs n$, then $t_{i-1,j}t_{(i-1)+1,l}\in J_2$.
If $i=1$, then $t_{1,l}t_{0,j}=t_{n,n+j}t_{n+1,l+n}\in J_2$. Hence, $J_1\subseteq J_2$. Similarly, $J_2\subseteq J_1$ and so $J_1=J_2$.
 \end{proof}

\begin{cor}(1)
  By the extension of modules, we have
  \begin{equation*}t_{ij}\in
    \begin{cases}
       [0,\min\{\az_i,a_{i+1,j}\}],~&\text{if}~|j-i|> 1;\\
      [0,\az_i],~&\text{if}~|j-i|=1;
    \end{cases}
  \end{equation*}
and for any $i=1,2,\cdots,n,~\sum\limits_{j>i}t_{ij}=\az_i$.

(2)The power of ${\bm q}$, $\sum\limits_{\stackrel{1\leqs i\leqs n}{i<l<j}}(a_{ij}t_{il}-t_{ij}t_{i+1,l})$, is non-negative.
\end{cor}
\begin{proof} Since $c_{ij}\geqs t_{ij}$, it follows from \eqref{cor1} that $a_{ij}\geqs t_{i-1,j}$, proving (1). (2) follows form \eqref{cor2}.
\end{proof}

\section{Distinguished words and a recursive formula}

For $A\in\ttz_\vartri^+(n)$, denote by $\ell (A)=\ell(M(A))$ the {\em Loewy length} of $M(A)$ and define the {\it periodicity} of $M(A)$ by
$$p(A)=\begin{cases} \max\{l\in\bbn\mid a_{i,i+l}\neq 0~\text{for all}~1\leqs i\leqs n\},&\text{ if $A$ is periodic}\\
0,&\text{ if $A$ is aperiodic}.\end{cases}$$
Clearly, $0\leqs p(A)\leqs\ell(A)$. Thus, $p(A)=0$ means that $A$ is aperiodic. If $p(A)=\ell(A)$, $A$ is called {\em strongly periodic}.

We now record several results in \cite{DengDuXiao2007generic} stated in multisegments in terms of matrices. Note that if $\Pi$ is the set of all multisegments, then there is a bijection 
$$\Pi\lra\ttz_\vartri^+(n), \pi=\sum_{i\in I,l\geqs 1}\pi_{i,l}[i;l)\longmapsto A_\pi=(a_{i,i+l})_{i\in I,l\geqs 1}\text{ with }a_{i,i+l}=\pi_{i,l}.$$

\begin{prop}[{\cite[\S4]{DengDuXiao2007generic}}]\label{dist words}
\begin{enumerate}[\rm(1)]
\item For any $A\in\ttz_\vartri^+(n)$, there exists uniquely a pair $(A',A'')$ associated with $A$ such that $A'$ is strongly periodic, $A''$ is aperiodic,  and $M(A)\cong M(A'')* M(A')$.

\item For aperiodic part $A''$, there exists a distinguished word $w_{A''}=j_1^{e_1}j_2^{e_2}\cdots j_t^{e_t}\in \ssz_I\cap \wp^{-1}(A'')$.

\item For strongly periodic part $A'$, there exists a distinguished word $w_{A'}=\bm a_1\bm a_2\cdots \bm a_p\in \ssz_{I^{\text{sin}}}\cap \wp^{-1}(A')$,
  moreover, $S_{\bm a_s}\cong \soc^{p-s+1}M(A')/\soc^{p-s}M(A'),1\leqs s\leqs p=p(A)$.

\item $w_{A''}w_{A'}=j_1^{e_1}j_2^{e_2}\cdots j_t^{e_t}\bm a_1\bm a_2\cdots \bm a_p$ is a distinguished word of $A$.
\end{enumerate}
\end{prop}

A construction of distinguished words of the strongly periodic part and aperiodic part has been given in \cite{DengDuXiao2007generic}. Building on this, we now
introduce some matrix algorithms to compute certain distinguished words in order to provide a monomial basis for computing the canonical basis.

If we take $A=(a_{i,j})$, then $M(A)=\oplus_{i=1}^n\oplus_{j>i}S_i[j-i]$ and $\soc(S_i[j-i])=S_{j-1}$, $\soc^2(S_i[j-i])=S_{j-2}[2],\cdots$, $\soc^l(S_i[j-i])=S_{j-l}[l]$. Here we understand $j-1\equiv j'(\mod n)$ and if $l\geqs j-i$, $\soc^l(S_i[j-i])=S_i[j-i]$.

We review the construction of producing the unique pair $(A',A'')$ in Proposition \ref{dist words}(1). For $A\in\ttz_\vartri^+(n)$ with $p=p(A)$,
then $\soc^p(M(A))=M(A')$ and $M(A'')\cong M(A)/M(A')$. 

\begin{defn}\label{A'A''}
For $A\in\ttz_\vartri^{+}(n)$ with $p=p(A)$,  define the {\it distinguished pair} $(A',A'')$ as follows.
\begin{enumerate}
\item The matrix $A'=(a'_{i,j})$, called the {\sf strongly periodic part} of $A$,  is obtained by setting
\begin{equation*}a'_{i,j}=
  \begin{cases}
    a_{i,j},&\text{if}~j<i+p,\\
    \sum\limits_{i_0\leqs i}a_{i_0,j},   &\text{if}~j=i+p.
  \end{cases}
\end{equation*}
In other words, $A'$ is the matrix obtained by replacing the ``$p^\th$-diagonal" $(a_{i,i+p})_{i\in \bbz}$ by $\col(B)$, where $B$ is the matrix obtained from $A$ by vanishing all the entries below the $p^\th$-diagonal.
\item The matrix $A''=(a''_{i,j})$, called the {\sf aperiodic part},  is obtained by setting
$$a''_{i,j}=a_{i,j+p}.$$
\end{enumerate}

\end{defn}

First, based on the structure of $\soc^t M(A)$ for strongly aperiodic $A\in\ttz_\vartri^{+}(n),t\in\bbn$, we give a matrix algorithm of \cite[Lemma 4.2]{DengDuXiao2007generic} as follows.

\begin{algo}[for the strongly periodic part]\label{algorithm for strongly periodic}
Suppose $A'$ is strongly periodic. Then $p=p(A')=\ell(A')$ and the algorithm runs $p$ steps:
\begin{itemize}
\item[]{\tt put $B=(b_{i,j}):=A'$\\
for $j$ from 1 to $p$ do\\
$T:=\sum_{i=1}^nb_{i,i+p-j+1}E_{i,i+p-j+1}$, $B:=B-T+\wit T^+$, $\bm a_j=\row(T)$ enddo\\
output \vspace{-2ex}
$$w_{A'}=\bm a_1\bm a_2\cdots\bm a_p.$$}
\end{itemize}
\end{algo}

\begin{rmk}
Every $\bm a_i$ is sincere and is uniquely determined by $A$. For $\lz=(\lz_i)_{i\in\bbz}\in\bbn_\vartri^n$, set $\lz^{[1]}=(\lz_i^{[1]})_{i\in\bbz}$, where $\lz_i^{[1]}=\lz_{i-1}$ for all $i\in\bbz$. It is easy to prove that there is one to one correspondence between strongly periodic matrix $A$ with $\ell(A)=p$ and a sincere sequence $\bm a_1\bm a_2\cdots\bm a_p$ with $\bm a_i^{[1]}\leqs \bm a_{i+1}$, for $1\leqs i\leqs p-1$.
\end{rmk}

Second, for $B=(b_{i,j})\in\ttz_\vartri^{ap}(n)$ and $i\in I$, we set $M(B)=\oplus_{i\in I}M_i(B)$ and $M_i(B)=\oplus_{j>i}b_{i,j}S_i[j-i]$. We take the maximal index in every step in \cite[Prop.~4.3]{DengDuXiao2007generic}, then we give the following matrix algorithm.

\begin{algo}[Aperiodic part]\label{algorithm for aperiodic part}
Suppose $A''$ is aperiodic with $l=\ell(A'')$, consider
{\tt
\begin{itemize}
\item[] put $B=(b_{i,j}):=A''$; for $i$ from 1 to $l$, do
\begin{itemize}\item[] if the $(l-i+1)th$ diagonal $b_{1,1+l-i+1},b_{2,2+l-i+1},\ldots,b_{n,n+l-i+1}$ is nonzero, choose the right most $b_{j,j+l-i+1}\neq0$ such that $b_{j+1,j+1+l-i+1}\neq0$; choose the minimal $j'\leq l-i+1$ such that $b_{j,j+j'}\neq0$ and $j'>\ell(M_{j+1}(B))$; do
$$T:=\sum_{k=j'}^{l-i+1}b_{j,j+k}E_{j,j+k},\quad B:=B-T+\wit T^+,\quad e_{i,j}:=\sum_{k=j'}^{l-i+1}b_{j,j+k},\quad \bm x_{i,j}=j^{e_{i,j}};$$ 
enddo; loop until the $(l-i+1)th$ diagonal is zero.
\end{itemize}
\item[]next $i$; enddo;\\
output \vspace{-2ex}
$$w_{A''}=\bm x_{1,j_1}\cdots\bm x_{1,j_a}\cdots\bm x_{l,k_1}\cdots\bm x_{l,k_b}$$
\end{itemize}}
\end{algo}
The two algorithms give a a {\em distinguished section}
\begin{equation}\label{dist sec}
\mathscr W(n)=\{w_A=w_{A''}w_{A'}\in\wp^{-1}(A)\cap\wit{\ssz}\mid A\in\ttz_\vartri^+(n)\}.
\end{equation}
When restricting to $\ttz_\vartri^{ap}(n)$, we obtain a distinguished section of $\ssz$ over $\ttz_\vartri^{ap}(n)$.

We explain the algorithms by the following example. Recall that every matrix in $\Theta_\vartri^{+}(n)$ is identified as its core. Sometimes, we indicate the diagonal with boldface entries for clarity.

\begin{exam}
  Suppose $n=3$ and $A=\begin{pmatrix}
    \bold 0 & 1 & 1 & 0 & 3 & 1 & 2 & 1 & 3\\
    0 & \bold0 & 0 & 2 & 3 & 1 & 0 & 1 & 1\\
    0 & 0 & \bold0 & 3 & 0 & 1 & 1 & 1 & 0\\
  \end{pmatrix}$,
then $p(A)=4,\ell(A)=8$ and
$$ A'=\begin{pmatrix}
    \bold0 & 1 & 1 & 0 & 6 & 0 & 0 \\
    0 & \bold0 & 0 & 2 & 3 & 6 & 0 \\
    0 & 0 & \bold0 & 3  & 0 & 1 & 3  \\
  \end{pmatrix},\quad A''=\begin{pmatrix}
    \bold0 & 1 & 2 & 1 & 3\\
    0 & \bold0 & 0 & 1 & 1\\
    0 & 0 & \bold0 & 1 & 0\\
  \end{pmatrix},$$
with $\ell(A')=\ell(A'')=4$.

Apply Algorithm \ref{algorithm for strongly periodic} to $A'$ gives
$$i=1:\quad T=\begin{pmatrix}
    \bold0 & 0 & 0 & 0 & 6 & 0 & 0 \\
    0 & \bold0 & 0 & 0 & 0 & 6 & 0 \\
    0 & 0 & \bold0 & 0 & 0 & 0 & 3  \\
  \end{pmatrix},\quad B=\begin{pmatrix}
    \bold0 & 1 & 1 & 3 & 0 & 0  \\
    0 & \bold0 & 0 & 2 & 9 & 0  \\
    0 & 0 & \bold0 & 3  & 0 & 7   \\
  \end{pmatrix},\quad \bm a_1=(6,6,3).$$
$$i=2:\quad T=\begin{pmatrix}
    \bold0 & 0 & 0 & 3 & 0 & 0  \\
    0 & \bold0 & 0 & 0 & 9 & 0  \\
    0 & 0 & \bold0 & 0 & 0 & 7   \\
  \end{pmatrix},\quad B=\begin{pmatrix}
    \bold0 & 1 & 8 & 0 & 0  \\
    0 & \bold0 & 0 & 5 & 0   \\
    0 & 0 & \bold0 & 3  & 9   \\
  \end{pmatrix},\quad\bm a_2=(3,9,7).$$
$$i=3:\quad T=\begin{pmatrix}
    \bold0 & 0 & 8 & 0 & 0  \\
    0 & \bold0 & 0 & 5 & 0  \\
    0 & 0 & \bold0 & 0 & 9   \\
  \end{pmatrix},\quad B=\begin{pmatrix}
    \bold0 & 10 & 0 & 0  \\
    0 & \bold0 & 8 & 0   \\
    0 & 0 & \bold0 & 8   \\
  \end{pmatrix},\quad \bm a_3=(8,5,9).$$
$$i=4: T=\begin{pmatrix}
    \bold0 & 10 & 0 & 0 \\
    0 & \bold0 & 8 & 0  \\
    0 & 0 & \bold0 & 8   \\
  \end{pmatrix},\quad B=0,\quad\bm a_4=(10,8,8).$$
The algorithm stops with the output $w_{A'}=\bm a_1\bm a_2\bm a_3\bm a_4.$

Applying Algorithm \ref{algorithm for aperiodic part} to $A''$ gives
$$i=1:\quad T=\begin{pmatrix}
    \bold0 & 0 & 0 & 0 & 3\\
    0 & \bold0 & 0 & 0 & 0\\
    0 & 0 & \bold0 & 0 & 0\\
  \end{pmatrix}, \quad B=\begin{pmatrix}
    \bold0 & 1 & 2 & 1 & 0\\
    0 & \bold0 & 0 & 1 & 4\\
    0 & 0 & \bold0 & 1 & 0\\
  \end{pmatrix},\quad \bm x_{1,1}=1^3.$$
$$i=2:\quad T=\begin{pmatrix}
    \bold0 & 0 & 0 & 0 & 0\\
    0 & \bold0 & 0 & 1 & 4\\
    0 & 0 & \bold0 & 0 & 0\\
  \end{pmatrix}, \quad B=\begin{pmatrix}
    \bold0 & 1 & 2 & 1 & 0\\
    0 & \bold0 & 0 & 0 & 0\\
    0 & 0 & \bold0 & 2 & 4\\
  \end{pmatrix},\quad \bm x_{2,2}=2^5.$$
$$\qquad\;T=\begin{pmatrix}
    \bold0 & 1 & 2 & 1 \\
    0 & \bold0 & 0 & 0 \\
    0 & 0 & \bold0 & 0 \\
  \end{pmatrix}, \quad A=\begin{pmatrix}
    \bold0 & 0 & 0 & 0 & 0\\
    0 & \bold0 & 2 & 1 & 0\\
    0 & 0 & \bold0 & 2 & 4\\
  \end{pmatrix},\quad \bm x_{2,1}=1^4.$$
$$i=3:\quad T=\begin{pmatrix}
    \bold0 & 0 & 0 & 0 \\
    0 & \bold0 & 0 & 0 \\
    0 & 0 & \bold2 & 4 \\
  \end{pmatrix}, \quad B=\begin{pmatrix}
    \bold0 & 4 & 0 & 0 & 0\\
    0 & \bold0 & 2 & 1 & 0\\
    0 & 0 & \bold0 & 0 & 0\\
  \end{pmatrix},\quad \bm x_{3,3}=3^6.$$
$$\qquad T=\begin{pmatrix}
    \bold0 & 0 & 0 & 0 \\
    0 & \bold0 & 2 & 1 \\
    0 & 0 & \bold0 & 0 \\
  \end{pmatrix}, \quad B=\begin{pmatrix}
    \bold0 & 4 & 0 & 0 \\
    0 & \bold0 & 0 & 0 \\
    0 & 0 & \bold0 & 1 \\
  \end{pmatrix},\quad \bm x_{3,2}=2^3.$$
$$i=4: T=\begin{pmatrix}
    \bold0 & 0 & 0 & 0 \\
    0 & \bold0 & 0 & 0 \\
    0 & 0 & \bold0 & 1 \\
  \end{pmatrix}, \quad B=\begin{pmatrix}
    \bold0 & 4 & 0 & 0 \\
    0 & \bold0 & 0 & 0 \\
    0 & 0 & \bold0 & 0 \\
  \end{pmatrix},\quad \bm x_{4,3}=3^1.$$
  $$\quad T=\begin{pmatrix}
    \bold0 & 4 & 0 & 0 \\
    0 & \bold0 & 0 & 0 \\
    0 & 0 & \bold0 & 0 \\
  \end{pmatrix}, \quad B=\begin{pmatrix}
    \bold0 & 0 & 0 & 0 \\
    0 & \bold0 & 0 & 0 \\
    0 & 0 & \bold0 & 0 \\
  \end{pmatrix},\quad \bm x_{4,1}=1^4.$$
The algorithm has output $w_{A''}=\bm x_{1,1}\bm x_{2,2}\bm x_{2,1} \bm x_{3,3} \bm x_{3,2}\bm x_{4,3}\bm x_{4,1}$.
Thus, it produces the following distinguished word associated to $A$ 
$$w_A=w_{A''}w_{A'}=\bm x_{1,1}\bm x_{2,2}\bm x_{2,1} \bm x_{3,3} \bm x_{3,2}\bm x_{4,3}\bm x_{4,1}\bm a_1\bm a_2\bm a_3\bm a_4.$$ 
\end{exam}
For a fixed  $A\in\ttz_\vartri^{+}(n)$, let 
\begin{equation}\label{poset ideal}
\Theta_A=(0,A]:=\{B\in\ttz_\vartri^{+}(n)\mid B\leqs_{\dg} A\} \text{ and }\Theta_{\prec A}=\{B\in\ttz_A\mid B\prec A\}.\end{equation}

The proof of Theorem \ref{Hall polynomials} shows that every $\varphi_{w_B,C}^{A}$ is divisible by $\varphi_{w_B}^{B}$. Let 
$\gamma_{w_B,C}^{A}=\varphi_{w_B,C}^{A}/\varphi_{w_B}^B$. The following result shows that the Hall polynomials
$\varphi_{B,C}^A$ can be computed by a recursive formula.
\begin{cor}\label{recursive formula}
For any $A,B,C\in\ttz_\vartri^+(n)$, let $w_B$ be the distinguished obtained by applying  Algorithms \ref{algorithm for strongly periodic} and
\ref{algorithm for aperiodic part} to $B$ and, for any $B'\leqs_{\dg}B$, let $\gamma^{B'}_{w_B}$ and $\gamma_{w_B,C}^A$ be obtained by the multiplication formula given in Theorem \ref{main-multip-theorem}. Then the Hall polynomial $\varphi_{B,C}^{A}$ can be computed by the recursive formula
$$\varphi_{B,C}^{A}=\begin{cases}
\gamma_{w_B,C}^A-\sum_{B':B'\prec B}\gamma^{B'}_{w_B}\varphi_{B',C}^{A},&\text{ if }A\in{\displaystyle \cup_{B'\prec B}\ttz_{\prec B'*C}};\\
\gamma_{w_B,C}^{A},&\text{ if }A\in\ttz_{B*C}\backslash\cup_{B'\prec B}\ttz_{\prec B'*C}.
\end{cases}$$
\end{cor}

\section{Ringel--Hall algebras, quantum affine $\mathfrak{gl}_n$ and their canonical bases}


The generic Hall algebra $\fkH_{\vartri}^\diamond(n)$ of $\ddz(n)$ is by definition the free $\bbz[{\bm q}]$-module with basis $\{u_A:=u_{[M(A)]}\mid A\in \ttz_{\vartri}^+(n)\}$ and multiplication given by
\begin{equation*}
  u_{B}\diamond u_{C}=\sum_{A\in\ttz_{\vartri}^+(n)}\vz^A_{B,C}u_A.
\end{equation*}
For a finite field $k$ of $q$ elements, by specializing $\bm q$ to $q$, we obtain the integral Hall algebra $\fkH_\vartri^\diamond(n,q)$ associated with $\Rep^0\ddz(n)$ discussed in \S\S2-4.

C.M. Ringel \cite{RInv, RRev} further twisted the multiplication, using the Euler form, to obtain the Ringel--Hall algebra which connects to the corresponding quantum group.

For ${\bf a}=(a_i)\in\bbz_\vartri^n$ and ${\bf b}=(b_i)\in\bbz_\vartri^n$, the Euler form associated
with the cyclic quiver $\ddz(n)$ is the bilinear form $\lan-,-\ran:\bbz_\vartri^n\times\bbz_\vartri^n\lra\bbz$ defined by
\begin{equation*}
  \lan {\bf a,b}\ran=\sum_{i\in I}a_ib_i-\sum_{i\in I}a_ib_{i+1}.
\end{equation*}

The (generic) Ringel-Hall algebra $\fkH_{\vartri}(n)$ of $\ddz(n)$ is by definition the algebra over $\cz=\bbz[v,v^{-1}]$ $(v^2=\bm q)$ with basis $\{u_A=u_{[M(A)]}\mid A\in \ttz_{\vartri}^+(n)\}$ and the multiplication is twisted by the Euler form:
\begin{equation*}
  u_{B}u_{C}=v^{\lan{\bf dim} M(B),{\bf dim} M(C)\ran}\sum_{A\in\ttz_{\vartri}^+(n)}\vz^A_{B,C}u_A.
\end{equation*}
It is well known that for two $A,B\in\ttz_\vartri^+(n)$, there holds
$$\lan{\bf dim} M(A),{\bf dim} M(B)\ran=\dim_k\Hom(M(A),M(B))-\dim_k\Ext^1_k(M(A),M(B)).$$

The $\cz$-subalgebra $\mathfrak{C}_\vartri(n)$ of $\fkH_\vartri(n)$ generated by $u_i^{(m)}=\dfrac{u_i^m}{[m]!},i\in I$ and $m\geqs 1$, is called the {\em (generic) composition subalgebra}. Then $\fkC_\vartri(n)$ is also generated by $u_{[mS_i]}$ since $u_i^{(m)}=v^{m(m-1)}u_{[mS_i]}$. Clearly, $\fkH_\vartri(n)\text{and}~\fkC_\vartri(n)$ admit natural $\bbn^n$-grading by dimension vectors:
$$\fkH_\vartri(n)=\bps_{{\bf d}\in\bbn^n}\fkH_\vartri(n)_{\bf d}\quad\text{and}\quad\fkC_\vartri(n)=\bps_{{\bf d}\in\bbn^n}\fkC_\vartri(n)_{\bf d}$$
where $\fkH_\vartri(n)_{\bf d}$ is spanned by all $u_A$ with ${\bf dim} M(A)={\bf d}$ and $\fkC_\vartri(n)_{\bf d}=\fkC_\vartri(n)\cap \fkH_\vartri(n)_{\bf d}$.

Base change gives the $\bbq(v)$-algebra ${\bm\fkH}_\vartri(n)=\fkH_\vartri(n)\otm_\cz\bbq(v)$ and ${\bm\fkC}_\vartri(n)=\fkC_\vartri(n)\otm_{\cz}\bbq(v)$.
Denote by ${\bm\fkH}^-_\vartri(n)$ the opposite algebra of ${\bm\fkH}^+_\vartri(n)$ ($={\bm\fkH}_\vartri(n)$). 

By extending $\bm\fkH_\vartri(n)$ to Hopf algebras
\begin{equation*}
  \bm\fkH_\vartri(n)^{\geqs0}={\bm\fkH}^+_\vartri(n)\otm\bbq(v)[K_1^{\pm1},\cdots,K_n^{\pm1}]~
  \text{and}~\bm\fkH_\vartri(n)^{\leqs0}=\bbq(v)[K_1^{\pm1},\cdots,K_n^{\pm1}]\otm {\bm\fkH}^-_\vartri(n),
\end{equation*}
we define the double Ringel-Hall algebra $\fkD_\vartri(n)$ (cf. \cite{Xiao1997drinfeld} \& \cite{DengDuFu2012double})
to be a quotient algebra of the free product $\bm\fkH_\vartri(n)^{\geqs0}*\bm\fkH_\vartri(n)^{\leqs0}$ via a certain skew Hopf paring $\psi:\bm\fkH_\vartri(n)^{\geqs0}\times\bm\fkH_\vartri(n)^{\leqs0}\ra\bbq(v)$. In particular, there is a triangular decomposition
$$\fkD_\vartri(n)=\fkD_{\vartri}^+(n)\otm \fkD_\vartri^0(n)\otm \fkD_\vartri^-(n),$$
where $\fkD_\vartri^+(n)\cong {\bm\fkH}^+_\vartri(n),\fkD_\vartri^0(n)\cong \bbq[K_1^{\pm1},\cdots,K_n^{\pm1}]$ and $\fkD_\vartri^-(n)\cong{\bm\fkH}^-_\vartri(n)$.

\begin{thm}[{\cite[Th. 2.5.3]{DengDuFu2012double}}]
  Let $\U_v(\wih{\fkg\fkl}_n)$ be the quantum loop algebra of $\fkg\fkl_n$ defined in \cite{Drinfeld1988new} or \cite[\S2.5]{DengDuFu2012double}. Then there is a Hopf algebra isomorphism $\fkD_\vartri(n)\cong \U_v(\wih{\fkg\fkl}_n)$.
\end{thm}

Let $\U=\U(n)=\U_v(\wih{\fks\fkl}_n)$ be the quantum affine $\fks\fkl_n(n\geqs 2)$ over $\bbq(v)$, and let $E_i,F_i,K_i^\pm(i\in I)$ be the generators, for
details see \cite{Lusztig1993introduction,Jantzen1995lectures}. Then $\U$ admits a triangular decomposition $\U=\U^-\U^0\U^+$, where $\U^+$(resp. $\U^-,\U^0$) is the subalgebra generated by the $E_i$ (resp. $F_i$, $K_i^\pm~(i\in I)$). Denote by $U_\cz^+$ the Lusztig integral form of $\U^+$, which is
generated by all the divided powers $E_i^{(m)}=\tfrac{E_i^m}{[m]!}$. The relation of Ringel-Hall algebras and quantum affine $\fks\fkl_n$ is described in the following.

\begin{thm}[{\cite{Ringel1993composition}}]\label{decomp of Hall alg and generators}
There is a $\cz$-algebra isomorphism $$\fkC_\vartri(n)\iso U_\cz^+(n),~u_i^{(m)}\mapsto E_i^{(m)},~i\in I,~m\geqs 1,$$
  and by base change to $\bbq(v)$, there is an algebra isomorphism $\bm\fkC_\vartri(n)\iso \U^+(n)$.
\end{thm}


We now review an algorithm for computing the canonicl basis. The first ingredient required in the algorithm is the following modified multiplication formulas.

For $A\in\ttz_\vartri^+(n)$, let $\dz(A)=\dim\End(M(A))-\dim M(A)$ and
$$\wit{u}_A=v^{\dz(A)}u_{A}=v^{\dim\End(M(A))-\dim M(A)}u_{A}.$$
\begin{lem}\label{DuFu}
For $\az\in\bbn_\vartri^n,A\in\ttz_\vartri^+(n)$, the twisted multiplication formula in the Ringel-Hall algebra $\fkH_\vartri(n)$ over $\cz$ is given by
\begin{equation*}
  \widetilde{u}_\az\widetilde{u}_A=\sum_{\stackrel{T\in \Theta_{\vartriangle}^+(n)}{\row(T)=\az}}v^{f_{A,T}}\prod_{\stackrel{1\leqs i\leqs n}{j\in\bbz,j>i}}
  \ol{\left[\!\!\left[\begin{matrix} a_{ij}+t_{ij}-t_{i-1,j}\\ t_{ij} \end{matrix}\right]\!\!\right]}\wit{u}_{A+T-\widetilde{T}^+},
\end{equation*}
where
\begin{equation*}
  f_{A,T}=\sum_{\stackrel{1\leqs i\leqs n}{j\geqs l\geqs i+1}}a_{i,j}t_{i,l}-\sum_{\stackrel{1\leqs i\leqs n}{j>l\geqs i+1}}a_{i+1,j}t_{i,l}
  -\sum_{\stackrel{1\leqs i\leqs n}{j\geqs l\geqs i+1}}t_{i-1,j}t_{i,l}+\sum_{\stackrel{1\leqs i\leqs n}{j>l\geqs i+1}}t_{i,j}t_{i,l}.
\end{equation*}
\end{lem}

For each $w={\bm a}_1{\bm a}_2\cdots{\bm a}_m\in\wit{\ssz}$ with tight form $w={\bm b}_1^{e_1}{\bm b}_2^{e_2}\cdots{\bm b}_t^{e_t}$,
define a monomial associated with $w$ in $\fkH_\vartri(n)$
$$m^{(w)}=\wit{u}_{e_1\bm b_1}\cdots \wit{u}_{e_t\bm b_t}.$$
The monomials associated with the distinguished words $w_A=w_{A''}w_{A'}$ produced by Algorithms \ref{algorithm for strongly periodic} and
\ref{algorithm for aperiodic part} will be denoted simply by
$$
m^{(A)}=m^{(w_A)}=m^{(w_{A''})}m^{(w_{A'})}.
$$
We now apply \cite[Thm~6.2]{DengDuXiao2007generic} to this particularly selected monomial set.
\begin{lem}\begin{enumerate}[\rm(1)]
\item For $A\in\ttz_\vartri^+(n)$,  we have a triangular relation
 \begin{equation}\label{triangluarrelation}
  m^{(A)}=\wit{u}_A
  +\sum_{\stackrel{T\prec A,T\in\ttz_\vartri^+(n)}{\bdim M(A)=\bdim M(T)}}v^{\dz(A)-\dz(T)}\gz^T_{w_A}(v^2)\wit{u}_T,
 \end{equation}
In particular, $\fkH_\vartri(n)$ is generated by $\{u_i^{(m)},u_\az=u_{[S_\az]}\mid i\in I,\az\in I_{\text{\rm sin}},m\in\bbn\}$,
where $S_\az=\oplus_{i=1}^n\az_i S_i$ is the semisimple representation of $\ddz(n)$ associated with $\az$.
\item  The set \begin{equation}\label{monomial basis}
\mathscr{M}(\wih{\fkg\fkl}_n)_+=\{m^{(A)}\mid A\in\ttz_\vartri^+(n)\}\quad(\text{resp.}, \mathscr{M}(\wih{\fkg\fkl}_n)_{ap}=\{m^{(A)}\mid A\in\ttz_\vartri^{ap}(n)\})
\end{equation}
forms  a $\cz$-basis for $\fkH_\vartri(n)$ (resp., $U_\cz^+(n)$).
\end{enumerate}
\end{lem}

The ingredients to define a canonical basis of an algebra include a basis with index set $P$, a bar involution on the algebra and a poset structure on $P$ which satisfies a certain triangular condition when applying the bar to a basis element. 
In the current case, the basis is the basis $\{\wit{u}_A\mid A\in\ttz_\vartri^{+}(n)\}$, the poset is $(\ttz_\vartri^{+}(n),\leqs_{\dg})$, and the bar involution (see, e.g., \cite[Proposition 7.5]{VaragnoloVasserot1999decomposition}) is given by 
$$^-:\fkH_\vartri(n)\lra \fkH_\vartri(n),\quad m^{(A)}\mapsto m^{(A)}, v\mapsto v^{-1}.$$
We now use the selected monomials $m^{(A)}$ to verify the triangular relation.

 Restricting to $A\in\ttz_\vartri^{+}(n)_{\bf d},~{\bf d}\in\bbn^n_\vartri$, by \eqref{triangluarrelation}
\begin{equation}\label{tri rel}
  m^{(A)}=\wit{u}_A
  +\sum_{B\prec A,B\in\ttz_\vartri^+(n)_{\bf d}}h_{B,A}\wit{u}_B,\quad h_{B,A}=v^{\dz(A)-\dz(T)}\gz^B_{w_A}(v^2).
\end{equation}
Solving above gives
\begin{equation*}
  \wit{u}_A=m^{(A)}+\sum_{B\prec A,B\in\ttz_\vartri^+(n)_{\bf d}}g_{B,A}m^{(B)}.
\end{equation*}
Applying the bar involution, we obtain
$$\ol{\wit{u}_A}=m^{(A)}+\sum_{B\in\ttz_\vartri^{+}(n)_{\bf d},B\prec A}\bar{g}_{B,A}m^{(B)}=\wit{u}_A+\sum_{B\in\ttz_\vartri^{+}(n)_{\bf d},B\prec A}r_{B,A}\wit{u}_B.$$

Now, by \cite[7.10]{Lusztig1990canonical} (or \cite[\S0.5]{DengDuParashallWang2008finite},\cite{Du1994ic}), the system
$$p_{B,A}=\sum_{B\preccurlyeq C\preccurlyeq A}r_{B,C}\bar{p}_{C,A}\quad \text{for}~B\preccurlyeq A, A,B\in\ttz_\vartri^{+}(n)_{\bf d}$$
has a unique solution satisfying $p_{A,A}=1,p_{B,A}\in v^{-1}\bbz[v^{-1}]$ for $B\prec A$. Moreover, the elements
$$\scc_A=\sum_{B\preccurlyeq A,B\in \ttz_\vartri^{ap}(n)}p_{B,A}\wit{u}_B,\quad A\in\ttz_\vartri^{+}(n)_{\bf d},$$
satisfying $\ol{\scc_A}=\scc_A$, form a $\cz$-basis for $\fkH_\vartri(n)_{\bf d}$. The basis 
\begin{equation}\label{CBS}
\mathscr{C}(\wih{\fkg\fkl}_n)_+=\{\scc_A\mid A\in\ttz_\vartri^{+}(n)\}
\end{equation} 
is called the {\it canonical basis} of $\fkH_\vartri(n)$ with respect to the PBW type basis $\{\wit{u}_A\}_{A\in\ttz_\vartri^{+}(n)}$, the bar involution and the poset $(\ttz_\vartri^{+}(n),\leqs_{\dg})$.

In practice, if the relation \eqref{tri rel} can be computed explicitly, then we may follow the following algorithm to compute the $\scc_A$ (or $p_{B,A}$) inductively on the poset ideal $\ttz_A$ defined in \eqref{poset ideal}. Write 
$$
\ttz_{\prec A}=\ttz_{\prec A}^{1}\cup \ttz_{\prec A}^{2}\cup\cdots\cup \ttz_{\prec A}^{t}~\text{for some}~t\in\bbn,
$$
where $\ttz_{\prec A}^{1}=\{\text{maximal elements of}~\ttz_{\prec A}\}$ and
$\ttz_{\prec A}^{i}=\{\text{maximal elements of}~\ttz_{\prec A}\setminus \cup_{j=1}^{i-1}\ttz_{\prec A}^j\}$ for $2\leqs i\leqs t$. Let 
$$'\ttz_{\prec A}^{a}=\{B\in \ttz_{\prec A}^{a}\mid h_{B,A}\not\in v^{-1}\bbz[v^{-1}]\}.$$

In the summation \eqref{tri rel}, assume $'\ttz_{\prec A}^{a}\not=\emptyset$ with $a$ minimal. Then $p_{B,A}:=h_{B,A}\in v^{-1}\bbz[v^{-1}]$ for all $B\in \ttz_{\prec A}^i$ with $i<a$ or $B\in \ttz_{\prec A}^a\backslash{}'\ttz_{\prec A}^a$.
For each $B\in{}'\ttz_{\prec A}^{a}$, $h_{B,A}\notin v^{-1}\bbz[v^{-1}]$ has a unique decomposition $h_{B,A}=h'_{B,A}+p_{B,A}$
with $\ol{h'_{B,A}}=h'_{B,A}$ and $p_{B,A}\in v^{-1}\bbz[v^{-1}]$.
Then
$$m^{(A)}-\sum_{B\in{}'\ttz_{\prec A}^a}h'_{B,A}m^{(B)}=\wit{u}_{A}+\sum_{B\in\ttz_{\prec A}^i,i\leqs a}p_{B,A}
\wit{u}_B+\sum_{\stackrel{B\in\ttz_{\prec A}^i}{i>a}}g_{B,A}\wit{u}_B.$$

Continue this argument with $g_{B,A}$ if necessary, we eventually obtain
$$m^{(A)}-\sum_{B\in{}'\ttz_{\prec A}}h'_{B,A}m^{(B)}\in\wit{u}_{A}+\sum_{{B\leqs_{\dg} A}\atop{B\in\ttz_{\vartri}^{+}(n)}}v^{-1}\bbz[v^{-1}]
\wit{u}_B,$$
where $'\ttz_{\prec A}$ is a union of those $'\ttz_{\prec A}^a$.
Since 
$$\ol{m^{(A)}-\sum_{B\in{}'\ttz_{\prec A}}h'_{B,A}m^{(B)}}=m^{(A)}-\sum_{B\in{}'\ttz_{\prec A}}h'_{B,A}m^{(B)},$$ by the uniqueness of the canonical basis of $\fkH_\vartri(n)$
with respect to the PBW type basis $\wit{u}_A$, we have proved the following.

\begin{algo}\label{another const for hall alg}
For $A\in\ttz_{\vartri}^{+}(n)$, there exist a recursively constructed subset $'\ttz_{\prec A}$ of $\ttz_A$ and elements
$h'_{B,A}\in\bbz[v,v^{-1}]$ for all $B\in {}'\ttz_{\prec A}$ such that $\ol{h'_{B,A}}=h'_{B,A}$ and
  $$\scc_A=m^{(A)}-\sum_{B\in{}'\ttz_{\prec A}}h'_{B,A}m^{(B)}$$
  is the canonical basis element associated with $A$.
\end{algo}

If ${}'\ttz_{\prec A}=\emptyset$, then $\scc_A=m^{(A)}$. Such a $\scc_A$ is call a {\it tight monomial}, following \cite{Lu93a}.

\section{Slices of the canonical basis}

In certain finite type cases, the canonical bases can be explicitly computed. See, for example, Lusztig \cite[\S3]{Lusztig1990canonical} for types $A_1$ and $A_2$ and \cite{XicanonicalA31999, XicanonicalB21999} for types $A_3$ and $B_2$. It is natural to expect that this is the case for quantum affine $\mathfrak{gl}_2$. However, this is much more complicated. In the next three sections, we present explicit formulas of the canonical basis for five ``slices''. We will see that if a module's Loewy length increases, the computation becomes more difficult.

The slices of the canonical basis is defined according to the Loewy length and periodicity of modules. In other words, for $(l,p)\in\bbn^2$ with $l\geq1,l\geq p\geq0$, let
$$\mathscr{C}(\wih{\fkg\fkl}_n)_{(l,p)}=\{\scc_A\mid \ell(A)=l,p(A)=p\}\quad(\text{resp.},\mathscr{M}(\wih{\fkg\fkl}_n)_{(l,p)}=\{m^{(A)}\mid \ell(A)=l,p(A)=p\}).$$
which is called a  {\it canonical ({\rm resp.,} monomial) slice}. Clearly, each of the bases is a disjoint union of slices. 

In the remaining paper, we will compute the slices $\mathscr{C}(\wih{\fkg\fkl}_2)^+_{(l,p)}$ for $l\leq 2$. We first compute
the cases for $(l,p)\in\{(1,0),(1,1),(2,0)\}$ which are relatively easy.

\begin{prop}\label{slice11} For $(l,p)=(1,0)$ or $(1,1)$, we have
$$\aligned
\mathscr{C}(\wih{\fkg\fkl}_2)_{(1,0)}&=\mathscr{M}(\wih{\fkg\fkl}_2)_{(1,0)}=\{\wit{u}_{aS_1},\wit{u}_{bS_2}\mid a,b\in\bbn-0\}\text{ and }\\
\mathscr{C}(\wih{\fkg\fkl}_2)_{(1,1)}&=\mathscr{M}(\wih{\fkg\fkl}_2)_{(1,1)}=\{\wit{u}_{aS_1\oplus bS_2}\mid a,b\in\bbn,ab\neq0\}.\endaligned
$$
\end{prop}

For $(l,p)=(2,0)$, all modules are aperiodic. If we put 
$$\mathscr{M}(\wih{\fkg\fkl}_n)_{ap}=\{m^{(A)}\mid A\in\ttz_\vartri^{ap}(n)\}$$
(cf. \eqref{monomial basis}),
then the structure of the monomial basis $\mathscr{M}(\wih{\fkg\fkl}_2)_{ap}$ for the $+$-part $U^+_\bbz(2)$ of quantum affine $\mathfrak{sl}_2$  has a very simple description.

A sequence $(a_1,a_2,\ldots,a_l)\in\bbn^l$ is called a {\it pyramidic} if there exists $k$, $1\leq k\leq l$, such that
$$a_1\leqs a_{2}\leqs\cdots \leqs a_k,\quad a_k\geqs a_{k+1}\geqs \cdots\geqs a_l.$$

We identify the positive part $U^+_{\mathcal Z}(n)$ with the composition algebra under the isomorphism $\fkC_\vartri(n)\iso U_\cz^+(2),~u_i^{(m)}\mapsto E_i^{(m)}$ as given in Theorem~\ref{decomp of Hall alg and generators}.
\begin{lem}\label{structure of monomial basis} We have
$$\mathscr{M}(\wih{\fkg\fkl}_2)_{ap}=\{E_{i}^{(a_1)}E_{i+1}^{(a_2)}E_{i}^{(a_3)}E_{i+1}^{(a_4)}\cdots E_{i'}^{(a_l)}\mid i\in\bbz_2,(a_1,a_{2},\ldots, a_l) \text{ is pyramidic, }\forall l\in\bbn\},$$
where $i'=i$ if $l$ is odd and $i'=i+1$ if $l$ is even.
\end{lem}
\begin{proof}
Applying Algorithm \ref{algorithm for aperiodic part} to $A\in\ttz_\vartri^{ap}(2)$, we know $m^{(A)}$ has the desired form. 

Conversely, for a given 
$$E(i,\bm a)=E_{i}^{(a_1)}E_{i+1}^{(a_2)}E_{i}^{(a_3)}E_{i+1}^{(a_4)}\cdots E_{k}^{(a_k)}\cdots E_{l}^{(a_l)}$$
where $$0< a_1\leqs a_{2}\leqs\cdots \leqs a_k,\quad a_k\geqs a_{k+1}\geqs \cdots\geqs a_l>a_{l+1}=0,$$
we construct an $A\in\ttz_\vartri^{ap}(2)$ such that $m^{(A)}=E(i,\bm a)$. Since there are 8 cases for $(i,k,l)$, we only prove the case where $(i,k,l)=(1,1,1)$. The proof for other cases is similar.

First, the matrix giving  $E_{1}^{(a_k)}\cdots E_{1}^{(a_l)}$ by the algorithm has the form
\begin{equation*}
  \begin{pmatrix}
    0 & a_k-a_{k+1} & a_{k+1}-a_{k+2} & \cdots & a_{l-1}-a_l & a_l\\
    0 & 0 & 0 &\cdots & 0 & 0
  \end{pmatrix}.
\end{equation*}
For $a_{k-1}$, there exists a unique $i_0\in\bbn$ such that $a_{k+i_0}\geqs a_{k-1}> a_{k+i_0+1}$, and so $a_{k+i_0}-a_{k+i_0+1}=(a_{k+i_0}-a_{k-1})+(a_{k-1}-a_{k+i_0+1})$.
Now, the matrix giving $E_{2}^{(a_{k-1})}E_{1}^{(a_k)}\cdots E_{1}^{(a_l)}$ has the form
\begin{equation*}
  \begin{pmatrix}
    0 & a_k-a_{k+1} & a_{k+1}-a_{k+2} & \cdots & a_{k+i_0}-a_{k-1} & 0& 0 &\cdots& 0 & 0\\
    0 & 0 & 0 &\cdots & 0 & 0 & a_{k-1}-a_{k+i_0+1}  & \cdots & a_{l-1}-a_l& a_l
  \end{pmatrix}.
\end{equation*}
Continuing this pattern for $a_{k-2},\cdots,a_2,a_1$ yields eventually the required matrix $A$.
\end{proof}

We take an example to illustrate the construction.
\begin{exam}Consider
  $$E_1^{(2)}E_2^{(3)}E_1^{(5)}E_2^{(8)}E_1^{(9)}E_2^{(6)}E_1^{(4)}E_2^{(3)}E_1^{(1)}.$$

First,  the matrix giving $E_1^{(9)}E_2^{(6)}E_1^{(4)}E_2^{(3)}E_1^{(1)}$ is
\begin{equation*}
  \begin{pmatrix}
    0 &9-6& 6-4 & 4-3 & 3-1 & 1\\
    0 & 0 & 0 & 0 & 0 & 0
  \end{pmatrix}=
  \begin{pmatrix}
    0 &3& 2 & 1 & 2 & 1\\
    0 & 0 & 0 & 0 & 0 & 0
  \end{pmatrix}.
\end{equation*}

Since $9>8>6$, then the matrix giving $E_2^{(8)}E_1^{(9)}E_2^{(6)}E_1^{(4)}E_2^{(3)}E_1^{(1)}$ is
\begin{equation*}
  \begin{pmatrix}
    0 &9-8& 0 & 0 & 0 & 0 & 0 & 0\\
    0 & 0 & 0 & 8-6 & 6-4 & 4-3 & 3-1 & 1
  \end{pmatrix}=
  \begin{pmatrix}
    0 &1& 0 & 0 & 0 & 0 & 0 & 0\\
    0 & 0 & 0 & 2 & 2 & 1 & 2 & 1
  \end{pmatrix}.
\end{equation*}

Due to $6>5>4$, then the matrix giving $E_1^{(5)}E_2^{(8)}E_1^{(9)}E_2^{(6)}E_1^{(4)}E_2^{(3)}E_1^{(1)}$ is
\begin{equation*}
  \begin{pmatrix}
    0 &9-8& 0 & 0 & 5-4 & 4-3 & 3-1 & 1\\
    0 & 0 & 0 & 8-6 & 6-5 & 0 & 0 & 0
  \end{pmatrix}=
  \begin{pmatrix}
    0 &1& 0 & 0 & 1 & 1 & 2 & 1\\
    0 & 0 & 0 & 2 & 1 & 0 & 0 & 0
  \end{pmatrix}.
\end{equation*}

Since $4> 3\geqs3$, then the matrix giving $E_2^{(3)}E_1^{(5)}E_2^{(8)}E_1^{(9)}E_2^{(6)}E_1^{(4)}E_2^{(3)}E_1^{(1)}$ is
\begin{equation*}
  \begin{pmatrix}
    0 &9-8& 0 & 0 & 5-4 & 4-3 & 0 & 0 & 0 & 0\\
    0 & 0 & 0 & 8-6 & 6-5 & 0 & 0 & 3-3 &3-1 & 1
  \end{pmatrix}=
  \begin{pmatrix}
    0 &1& 0 & 0 & 1 & 1 & 0 & 0 &0 & 0\\
    0 & 0 & 0 & 2 & 1 & 0 & 0 & 0 &2 & 1
  \end{pmatrix}.
\end{equation*}

Finally, since $3>2>1$, the matrix giving $E_1^{(2)}E_2^{(3)}E_1^{(5)}E_2^{(8)}E_1^{(9)}E_2^{(6)}E_1^{(4)}E_2^{(3)}E_1^{(1)}$ has the form
\begin{equation*}
  \begin{pmatrix}
    0 &9-8& 0 & 0 & 5-4 & 4-3 & 0 & 0 & 2-1 & 1\\
    0 & 0 & 0 & 8-6 & 6-5 & 0 & 0 & 3-3 &3-2 & 0
  \end{pmatrix}=
  \begin{pmatrix}
    0 &1& 0 & 0 & 1 & 1 & 0 & 0 & 1 & 1\\
    0 & 0 & 0 & 2 & 1 & 0 & 0 & 0 & 1 & 0
  \end{pmatrix}.
\end{equation*}
\end{exam}
\vspace{.3cm}
Now we are ready to describe the slice $\mathscr{C}(\widehat{\mathfrak{gl}}_2)_{(2,0)}$ which is similar to the slices in Proposition \ref{slice11}.


\begin{prop}\label{slice20}
For $(l,p)=(2,0)$, we have
$$
\mathscr{C}(\wih{\fkg\fkl}_2)_{(2,0)}=\mathscr{M}(\wih{\fkg\fkl}_2)_{(2,0)}=\{E_1^{(a+b)}E_2^{(b)},E_2^{(b)}E_1^{(a+b)},E_1^{(b)}E_2^{(a+b)},E_2^{(a+b)}E_1^{(b)}\mid a,b\in\bbn, b>0\}.$$
\end{prop}
\begin{proof}Suppose $A\in\ttz_\vartri^+(2)$ with $(\ell(A),p(A))=(2,0)$, then $A$ is one of the following matrices
\begin{equation*}
  \begin{pmatrix}
    \bm0 & a & b\\0 & \bm0 & 0
  \end{pmatrix},\quad\begin{pmatrix}
    \bm0 & 0 & 0 & 0\\0 & \bm0 & a & b
  \end{pmatrix},\quad\begin{pmatrix}
   \bm 0 & 0 & b\\0 & \bm0 & a
  \end{pmatrix},\quad\begin{pmatrix}
    \bm0 & a & 0 &0\\0 & \bm0 & 0 & b
  \end{pmatrix},\quad\forall~a,b\in\bbn, b>0.
\end{equation*}
Applying Algorithm \ref{algorithm for aperiodic part} to these matrices or by Lemma \ref{structure of monomial basis},
the monomial $m^{(A)}$ has the following form
$$E_1^{(a+b)}E_2^{(b)},\;\;E_2^{(a+b)}E_1^{(b)},\;\;E_1^{(b)}E_2^{(a+b)},\;\;E_2^{(b)}E_1^{(a+b)}.$$
We now prove that these monomials are tight monomials. We only look at the first case, the other cases are similar.
We now apply the formula in Lemma \ref{DuFu} to compute 
$$m^{\begin{psmallmatrix}
     0 & a & b\\
     0 & 0 & 0
   \end{psmallmatrix}}=E_1^{(a+b)}E_2^{(b)}=\wit{u}_{(a+b)S_1}\wit{u}_{bS_2}.$$
Since $\alpha=(a+b,0)$, the matrix $T$ in the sum must be of the form  $\begin{psmallmatrix}
  0 & a+b-t & t\\
  0 & 0 & 0
\end{psmallmatrix}.$
Thus,
\begin{equation*}
  \begin{split}
   m^{\begin{psmallmatrix}
     0 & a & b\\
     0 & 0 & 0
   \end{psmallmatrix}}&=\sum_{t\leqs b}v^{-(a+b-t)(b-t)}\wit{u}_{\begin{psmallmatrix}
  0 & a+b-t & t\\
  0 & 0 & b-t
\end{psmallmatrix}}\\
&=\wit{u}_{\begin{psmallmatrix}
  0 & a & b\\
  0 & 0 & 0
\end{psmallmatrix}}+\sum_{t<b}v^{-(a+b-t)(b-t)}\wit{u}_{\begin{psmallmatrix}
  0 & a+b-t & t\\
  0 & 0 & b-t
\end{psmallmatrix}},
  \end{split}
\end{equation*}
which is the canonical basis element associated to $\begin{psmallmatrix}
     0 & a & b\\
     0 & 0 & 0
   \end{psmallmatrix}$, since $v^{-(a+b-t)(b-t)}\in v^{-1}\bbz[v^{-1}]$ for all $t<b$.
\end{proof}

In the three slices above, the recursively constructed subset $'\ttz_{\prec A}$ in Algorithm \ref{another const for hall alg} is empty. So they consist of tight monomials.

\section{Computing the slice $\mathscr{C}(\widehat{\mathfrak{gl}}_2)_{(2,1)}$}

For computing the slices $\mathscr{C}(\widehat{\mathfrak{gl}}_2)_{(2,1)}$ and $\mathscr{C}(\widehat{\mathfrak{gl}}_2)_{(2,2)}$ in this and next sections, we consider a matrix of the form 
\begin{equation}\label{matA}
A=\begin{pmatrix}
  0 & a & c & 0\\
  0 & 0 & b & d
\end{pmatrix}\in\ttz_\vartri^+(2)
\end{equation}
satisfying $\ell(A)=2,p(A)>0$, where $a,b,c,d\in\bbn$. Then $c+d\neq0$ and $ab+cd\neq0$.

\begin{lem} \label{poset ideal1}
For the $A$ as given above, we have
$$\ttz_A=(0,A]=\{A_{(k_1,k_2)}\mid k_1,k_2\in \bbn,(k_1,k_2)\leqs (c,d)\},$$ where $$A_{(k_1,k_2)}=\begin{pmatrix}
  0 & a+c+d-k_1-k_2 & k_1 &0\\0 &  0 & b+c+d-k_1-k_2 & k_2
\end{pmatrix}.$$
\end{lem}
\begin{proof} The proof is straightforward by \eqref{dg order}. Note also that
$$A_{(t_1,t_2)}\leqs_{\dg}A_{(k_1,k_2)}\iff (t_1,t_2)\leqs (k_1,k_2).\vspace{-4ex}$$
\end{proof}

For $cd\neq0$, the poset ideal can be described by its Hasse diagram:
\begin{center}
\begin{tikzpicture}
\node [above] at (0,0) {$(c,d)$};
\node [above,right] at (1,-0.8) {$(c,d-1)$};
\node [above,left] at (-1,-0.8) {$(c-1,d)$};
\node [above,left] at (-2,-1.8) {$(c-2,d)$};\node [above] at (0,-1.8) {$(c-1,d-1)$};\node [above,right] at (2,-1.8) {$(c,d-2)$};
\node [right] at (5,-5) {$(c,0)$};\node [left] at (-6,-6) {$(0,d)$};\node [below] at (-1,-11) {$(0,0)$};
\draw[fill] (0,0) circle [radius=0.025];\draw[fill] (1,-1) circle [radius=0.025];\draw[fill] (-1,-1) circle [radius=0.025];
\draw[fill] (-2,-2) circle [radius=0.025];\draw[fill] (0,-2) circle [radius=0.025];\draw[fill] (2,-2) circle [radius=0.025];
\draw[fill] (-3,-3) circle [radius=0.025];\draw[fill] (-1,-3) circle [radius=0.025];\draw[fill] (1,-3) circle [radius=0.025];
\draw[fill] (3,-3) circle [radius=0.025];
\draw (0,0)--(1,-1);\draw (0,0)--(-1,-1);
\draw (-1,-1)--(-2,-2);\draw (-1,-1)--(0,-2);\draw (1,-1)--(0,-2);\draw (1,-1)--(2,-2);
\draw (-2,-2)--(-3,-3);\draw (0,-2)--(-1,-3);\draw (-2,-2)--(-1,-3);\draw (0,-2)--(1,-3);\draw (2,-2)--(1,-3);\draw (2,-2)--(3,-3);
\draw [dotted](-3,-3)--(-6,-6);\draw [dotted](-1,-3)--(-5,-7);\draw [dotted](-5,-5)--(0,-10);\draw[dotted] (-6,-6)--(-1,-11);
\draw [dotted](3,-3)--(5,-5);\draw [dotted](3,-3)--(-3,-9); \draw[dotted] (5,-5)--(-1,-11); \draw [dotted](1,-3)--(4,-6);
\draw [dotted](4,-4)--(-2,-10);\draw [dotted](-3,-3)--(2,-8);\draw [dotted](-1,-3)--(3,-7);
\draw [dotted](1,-3)--(-4,-8);\draw [dotted](-4,-4)--(1,-9);
\end{tikzpicture}
\\
{\bf Figure} $H(c,d)$\\
\end{center}

For $B=A_{(k_1,k_2)}$, by Definition \ref{A'A''}, we have $B'=\begin{psmallmatrix}
    0 & a+c+d-k_1 & 0\\ 0 & 0 &b+c+d-k_2
  \end{psmallmatrix}$ and $B''=\begin{psmallmatrix}
    0 & k_1 & 0\\
    0 & 0 & k_2
  \end{psmallmatrix}$. 
The following follows immediately from Lemma \ref{DuFu}.

\begin{lem}\label{mk1k2}
Putting $\wit{u}_{(k_1,k_2)}=\wit{u}_{A_{(k_1,k_2)}}$ and $m^{(k_1,k_2)}=m^{(A_{(k_1,k_2)})}$, we have
\begin{equation*}
 \begin{split}
m^{(k_1,k_2)}&=\wit{u}_{\begin{psmallmatrix}
    0 & k_1 & 0\\
    0 & 0 & k_2
  \end{psmallmatrix}}\wit{u}_{\begin{psmallmatrix}
    0 & a+c+d-k_1 & 0\\ 0 & 0 &b+c+d-k_2
  \end{psmallmatrix}}\\
  &=\sum_{t_1\leqs k_1,t_2\leqs k_2}v^{(a-b-k_1+k_2+t_1-t_2)(k_1-k_2-t_1+t_2)}\ol{\left[\!\!\left[\begin{matrix} a+c+d-t_1-t_2\\ k_1-t_1\end{matrix}\right]\!\!\right]}\ol{\left[\!\!\left[\begin{matrix} b+c+d-t_1-t_2\\ k_2-t_2\end{matrix}\right]\!\!\right]}\wit{u}_{(t_1,t_2)}.\\
 \end{split}
\end{equation*}
\end{lem}

We now compute the canonical basis elements for those $A$ with  $c=0$ or $d=0$ (but not both zero). In other words, $p(A)=1$. We need the following identities for symmetric Gaussian polynomials.
\begin{lem}[{\cite[Section 3.1]{XicanonicalA31999}}]\label{technic-lem-of-gauss-poly}
\begin{enumerate}[\rm(1)]
 \item Assume that $m\geqs k\geqs 0,\dz\in\bbn$. Then
  \begin{equation*}
    \sum_{i=0}^\dz(-1)^iv^{i(m-k)}\begin{bmatrix}
    k-1+i\\k-1
  \end{bmatrix}\begin{bmatrix}
    m\\ \dz-i
  \end{bmatrix}=v^{-k\dz}\begin{bmatrix}
    m-k\\\dz
  \end{bmatrix}.
  \end{equation*}
\item Assume that $m\geqs k\geqs 0,\dz,n\in\bbn$. Then
  \begin{equation*}
    \sum_{i=0}^\dz(-1)^iv^{i(m-k-n)}\begin{bmatrix}
    k-1+i\\k-1
  \end{bmatrix}\begin{bmatrix}
    m+n\\ \dz-i
  \end{bmatrix}=\sum_{t=0}^{\text{\rm min}\{\dz,n\}}v^{-k(\dz-t)-n\dz+t(m+n)}\begin{bmatrix}
    m-k\\\dz-t
  \end{bmatrix}\begin{bmatrix}
    n\\t
  \end{bmatrix}.
  \end{equation*}
\end{enumerate}
\end{lem}

We now perform Algorithm \ref{another const for hall alg} to compute the slice $\mathscr{C}(\widehat{\mathfrak{gl}}_2)_{(2,1)}$. In this case, the recursively constructed subset $'\ttz_{\prec A}$ in the Algorithm \ref{another const for hall alg} is $'\ttz_{\prec A}=\ttz_{\prec A}$.

\begin{thm}\label{theorem for case (2,1)}
If $A\in\ttz_\vartri^+(2)$ with $(\ell(A),p(A))=(2,1)$, then $A$ is of the form
\begin{equation*}
  \begin{pmatrix}
    0 & a & c\\0 & 0 & b
  \end{pmatrix}\quad\text{ or }\quad 
  \begin{pmatrix}
    0 & a & 0 & 0\\0 & 0 & b & d
  \end{pmatrix}~(a,b,c,d\in\bbn_{\geqs1}).
\end{equation*}
\begin{enumerate}[\rm(1)]
\item For $A= \begin{pmatrix}
    0 & a & c\\0 & 0 & b
  \end{pmatrix}$, $\scc_A=m^{(A)}$ is a tight monomial if and only if $a\leqs b$. The canonical basis element associated to $A$ with $a>b$ has the form
$$\scc_{A}=\sum_{k=0}^c(-1)^{c-k}\begin{bmatrix}
  a-b-1+c-k\\a-b-1
\end{bmatrix}m^{(k,0)}=\sum_{t=0}^cv^{-t(a+t)}\begin{bmatrix}
  b+t\\t
\end{bmatrix}\wit{u}_{(c-t,0)},$$
where $\wit{u}_{(k,0)}=\wit{u}_{A_{(k,0)}}$ and $A_{(k,0)}=\begin{pmatrix}
    0 & a+c-k & k\\0 & 0 & b+c-k
  \end{pmatrix}.$
\item For $A=\begin{pmatrix}
    0 & a & 0 & 0\\0 & 0 & b & d
  \end{pmatrix}$, $\scc_A=m^{(A)}$ is a tight monomial if and only if $a\geqs b$. The canonical basis element associated to $A$ with $a<b$ has the form
$$\scc_{A}=\sum_{l=0}^d(-1)^{d-l}\begin{bmatrix}
  b-a-1+d-l\\b-a-1
\end{bmatrix}m^{(0,l)}=\sum_{t=0}^dv^{-t(b+t)}\begin{bmatrix}
  a+t\\t
\end{bmatrix}\wit{u}_{(0,d-t)},$$
where $\wit{u}_{(0,l)}=\wit{u}_{A_{(0,l)}}$ and $A_{(0,l)}=\begin{pmatrix}
    0 & a+d-l & 0&0\\0 & 0 & b+d-l & l
  \end{pmatrix}.$
\end{enumerate}
\end{thm}

\begin{proof}We only prove (1); the proof for (2) is similar. In this case, Hasse diagram $H(c,0)$ is a linear figure. In other words, we have $A=A_{(c,0)}>_{\dg}A_{(c-1,0)}>_{\dg}\cdots>_{\dg}A_{(1,0)}>_{\dg}A_{(0,0)}$. Note in this case that 
$A'=\begin{pmatrix}
    0 & a & 0\\ 0 & 0 &b+c
  \end{pmatrix}$ and $A''=\begin{pmatrix}
    0 & c\\
    0 & 0
  \end{pmatrix}$. Thus,
 $$ m^{(A)}=m^{(A'')}m^{(A')}=\wit{u}_{A''}\wit{u}_{A'}.$$
We now apply formula in Lemma \ref{DuFu}. We have here $\alpha=(c,0)$. It $T\in\ttz_\vartri^+(2)$ satisfying $A'-T+\widetilde T^+\in\ttz_\vartri^+(2)$ and $\row(T)=\alpha$, then $T=\begin{pmatrix}
    0 & c-t & t\\ 0 & 0 &0
  \end{pmatrix}$ for some $0\leq t\leq c$. Thus, $A'-T+\widetilde T^+=A_{(t,0)}$ and 
  $$f_{A',T}=a(c-t)-(b+c)(c-t)+c(c-t)=(c-t)(a-b-c+t).$$
Hence,
\begin{equation*}
 \begin{split}
  m^{(A)}&=\sum_{0\leqs t\leqs c}v^{(c-t)(a-b-c+t)}\ol{\left[\!\!\left[\begin{matrix} a+c-t\\ c-t\end{matrix}\right]\!\!\right]}\wit{u}_{(t,0)}=\wit{u}_{A}+\sum_{0\leqs t\leqs c-1}v^{(c-t)(a-b-c+t)}\ol{\left[\!\!\left[\begin{matrix} a+c-t\\ c-t\end{matrix}\right]\!\!\right]}\wit{u}_{(t,0)}.
 \end{split}
\end{equation*}
Consequently, $m^{(A)}$ becomes a canonical basis element (or a tight monomial) if $a\leqs b$.

By the calculation above, we have, for $k=0,1,2,\cdots,c$, $A_{(k,0)}=\begin{psmallmatrix}
    0 & a+c-k & k\\0 & 0 & b+c-k
  \end{psmallmatrix}$ and so, by Lemma~\ref{mk1k2},
\begin{equation*}
\begin{split}
  m^{(k,0)}=\wit{u}_{\begin{psmallmatrix}
    0 & k\\
    0 & 0
  \end{psmallmatrix}}\wit{u}_{\begin{psmallmatrix}
    0 & a+c-k & 0\\ 0 & 0 &b+c
  \end{psmallmatrix}}&=\sum_{0\leqs t\leqs k}v^{(k-t)(a+t-b-k)}\ol{\left[\!\!\left[\begin{matrix} a+c-t\\ k-t\end{matrix}\right]\!\!\right]}\wit{u}_{(t,0)}\\
  &=\sum_{0\leqs t\leqs k}v^{(k-t)(t-b-c)}\left[\begin{matrix} a+c-t\\ k-t\end{matrix}\right]\wit{u}_{(t,0)}.
\end{split}
\end{equation*}
Assume now $a>b$ and consider the following bar fixed sum
\begin{equation*}
  \begin{split}
 M(c)  &:=\sum_{k=0}^c(-1)^{c-k}\begin{bmatrix}
  a-b-1+c-k\\a-b-1
  \end{bmatrix}m^{(k,0)}\\
  &=\sum_{k=0}^c(-1)^{c-k}\begin{bmatrix}
  a-b-1+c-k\\a-b-1
\end{bmatrix}\bigg(\sum_{t=0}^kv^{(k-t)(t-b-c)}\begin{bmatrix}
  a+c-t\\k-t
\end{bmatrix}\wit{u}_{(t,0)}\bigg)\\
&=\sum_{k=0}^c\sum_{t=0}^k(-1)^{c-k}v^{(k-t)(t-b-c)}\begin{bmatrix}
  a-b-1+c-k\\a-b-1
\end{bmatrix}\begin{bmatrix}
  a+c-t\\k-t
\end{bmatrix}\wit{u}_{(t,0)}\\
&=\sum_{t=0}^c\bigg(\sum_{k=t}^c(-1)^{c-k}v^{(k-t)(t-b-c)}\begin{bmatrix}
  a-b-1+c-k\\a-b-1
\end{bmatrix}\begin{bmatrix}
  a+c-t\\k-t
\end{bmatrix}\bigg)\wit{u}_{(t,0)}\\
&=\wit{u}_{A}+\sum_{t=0}^{c-1}\bigg(\sum_{k=t}^c(-1)^{c-k}v^{(k-t)(t-b-c)}\begin{bmatrix}
  a-b-1+c-k\\a-b-1
\end{bmatrix}\begin{bmatrix}
  a+c-t\\k-t
\end{bmatrix}\bigg)\wit{u}_{(t,0)}.\\
 \end{split}
\end{equation*}
However, for fixed $t$, 
$$\aligned
f_{(t,0)}&:=\sum_{k=t}^c(-1)^{c-k}v^{(k-t)(t-b-c)}\begin{bmatrix}
  a-b-1+c-k\\a-b-1
\end{bmatrix}\begin{bmatrix}
  a+c-t\\k-t
\end{bmatrix}\\
&=\sum_{k'=0}^{c'}(-1)^{c'-k'}v^{-k'(b+c')}\begin{bmatrix}
  a-b-1+c'-k'\\a-b-1
\end{bmatrix}\begin{bmatrix}
  a+c'\\k'
\end{bmatrix}\;(c'=c-t, k'=k-t)\\
&=v^{-c'(b+c')}\sum_{i=0}^{c'}(-1)^{i}v^{i(b+c')}\begin{bmatrix}
  a-b-1+i\\a-b-1
\end{bmatrix}\begin{bmatrix}
  a+c'\\c'-i
\end{bmatrix}\;(i=c'-k').
\endaligned$$
Let $k=a-b$, $m=a+c'$ and $\delta=c'$. Applying Lemma \ref{technic-lem-of-gauss-poly}(1) gives
$$f_{(t,0)}=v^{-c'(b+c')}v^{-c'(a-b)}\begin{bmatrix}
  b+c'\\c'
\end{bmatrix}=v^{-c'(a+c')}\begin{bmatrix}
  b+c'\\c'
\end{bmatrix}=v^{-c'(a-b+c')}\ol{\left[\!\!\left[\begin{matrix}
  b+c'\\c'
\end{matrix}\right]\!\!\right]}\in v^{-1}\bbz[v^{-1}],
$$
since $a>b$.
Hence, $
M(c)\in \wit{u}_{A}+\sum_{t=0}^{c-1}v^{-1}\bbz[v^{-1}]\wit{u}_{(t,0)}.$ On the other hand, $\ol{M(c)}=M(c)$. Consequently, $\scc_A=M(c)$, as desired.
\end{proof}

\section{Computing the slice $\mathscr{C}(\widehat{\mathfrak{gl}}_2)_{(2,2)}$}

In the last part of this section, we show the canonical basis associated to the matrix $A=\begin{pmatrix}
  0 & a & c & 0\\
  0 & 0 & b & d
\end{pmatrix}$ with $\ell(A)=2=p(A)$ and $a,b,c,d\in\bbn$. Thus, $cd\neq0$.

\begin{thm}\label{slice22}
Maintain the notation as set in Lemmas \ref{poset ideal1} and \ref{mk1k2}.
Suppose $A=\begin{pmatrix}
  0 & a & c & 0\\
  0 & 0 & b & d
\end{pmatrix}\in\ttz_\vartri^+(2)$ with $\ell(A)=2=p(A)$ and $a,b,c,d\in\bbn$. Then the canonical basis element $\scc_A$ associated to $A$ is given as follows.
\begin{enumerate}[\rm(1)]
\item If $a=b$, then 
$$\scc_A=m^{(c,d)}-m^{(c-1,d-1)}.$$
\item If $a>b$, then
\begin{equation*}
  \scc_A=\sum_{k_1=0}^c(-1)^{c-k_1}\begin{bmatrix}
  a-b-1+c-k_1\\a-b-1
\end{bmatrix}m^{(k_1,d)}-\sum_{l_1=0}^{c-1}(-1)^{c-1-l_1}\begin{bmatrix}
  a-b-2+c-l_1\\a-b-1
\end{bmatrix}m^{(l_1,d-1)}.
\end{equation*}
\item If $a<b$, then
\begin{equation*}
  \scc_A=\sum_{k_1=0}^d(-1)^{d-k_1}\begin{bmatrix}
  b-a-1+d-k_1\\b-a-1
\end{bmatrix}m^{(c,k_1)}-\sum_{l_1=0}^{d-1}(-1)^{d-1-l_1}\begin{bmatrix}
  b-a-2+d-l_1\\b-a-1
\end{bmatrix}m^{(c-1,l_1)}.
\end{equation*}
\end{enumerate}
\end{thm}
We may see the symmetry of the three cases from the big diamond Figure $H(c,d)$:  The recursively constructed subset in Algorithm \ref{another const for hall alg} has the form:
$$'\ttz_{\prec A}=\begin{cases}\{A_{(c-1,d-1)}\},&\text{ in (1)};\\
\{A_{(i,d)},A_{(j,d-1)}\mid 0\leqs i,j\leqs c, i<c\},&\text{ in (2)};\\
\{A_{(c,i)},A_{(c-1,j)}\mid 0\leqs i,j\leqs d, i<d\},&\text{ in (3)}.
\end{cases}$$
\begin{proof}
We first prove (1) and thus assume $a=b$. Then the formula in Lemma \ref{mk1k2} with $(k_1,k_2)=(c,d)$ becomes
\begin{equation*}
\begin{split}
 m^{(c,d)}=&\sum_{t_1\leqs c,t_2\leqs d}v^{-(c-d-t_1+t_2)^2}\ol{\left[\!\!\left[\begin{matrix} a+c+d-t_1-t_2\\ c-t_1\end{matrix}\right]\!\!\right]}\ol{\left[\!\!\left[\begin{matrix} a+c+d-t_1-t_2\\ d-t_2\end{matrix}\right]\!\!\right]}\wit{u}_{(t_1,t_2)}\\
 =&\sum_{\substack{t_1\leqs c,t_2\leqs d\\c-t_1=d-t_2}}\ol{[\![\begin{smallmatrix} a+c+d-t_1-t_2\\ c-t_1\end{smallmatrix}]\!]}^2\wit{u}_{(t_1,t_2)}+\sum_{\substack{t_1\leqs c,t_2\leqs d\\c-t_1\neq d-t_2}}v^{-(c-t_1-d+t_2)^2}\ol{[\![\begin{smallmatrix} a+c+d-t_1-t_2\\ c-t_1\end{smallmatrix}]\!]}\ol{[\![\begin{smallmatrix} a+c+d-t_1-t_2\\ d-t_2\end{smallmatrix}]\!]}\wit{u}_{(t_1,t_2)}.
\end{split}
\end{equation*}
Since $\ol{[\![\begin{smallmatrix} a+c+d-t_1-t_2\\ c-t_1\end{smallmatrix}]\!]}^2-1\in v^{-1}\bbz[v^{-1}]$ ($=0$ if $(t_1,t_2)=(c,d)$) and the coefficients in the second sum are all in $v^{-1}\bbz[v^{-1}]$, it follows that
\begin{equation*}
 m^{(c,d)}=\begin{cases}\wit{u}_{(c,d)}+\wit{u}_{(c-1,d-1)}+\cdots+\wit{u}_{(c-d,0)}+X,&\text{ if }c\geq d;\\
 \wit{u}_{(c,d)}+\wit{u}_{(c-1,d-1)}+\cdots+\wit{u}_{(0,d-c)}+Y,&\text{ if }c<d,
 \end{cases}
\end{equation*}
where $X,Y\in \sum_{(t_1,t_2)<(c,d)}v^{-1}\bbz[v^{-1}]\wit{u}_{(t_1,t_2)}.$

Similarly, we have
\begin{equation*}
\begin{split}
 m^{(c-1,d-1)}=&\sum_{t_1\leqs c-1,t_2\leqs d-1}v^{-(c-d-t_1+t_2)^2}\ol{[\![\begin{smallmatrix} a+c+d-t_1-t_2\\ c-1-t_1\end{smallmatrix}]\!]}\ol{[\![\begin{smallmatrix} a+c+d-t_1-t_2\\ d-1-t_2\end{smallmatrix}]\!]}\wit{u}_{(t_1,t_2)}\\
 =&\sum_{\substack{t_1\leqs c,t_2\leqs d\\c-t_1=d-t_2}}\ol{[\![\begin{smallmatrix} a+c+d-t_1-t_2\\ c-1-t_1\end{smallmatrix}]\!]}^2\wit{u}_{(t_1,t_2)}+\sum_{\substack{t_1\leqs c,t_2\leqs d\\c-t_1\neq d-t_2}}v^{-(c-d+t_1-t_2)^2}\ol{[\![\begin{smallmatrix} a+c+d-t_1-t_2\\ c-1-t_1\end{smallmatrix}]\!]}\ol{[\![\begin{smallmatrix} a+c+d-t_1-t_2\\ d-1-t_2\end{smallmatrix}]\!]}\wit{u}_{(t_1,t_2)}.
\end{split}
\end{equation*}
and
\begin{equation*}
 m^{(c-1,d-1)}=\begin{cases}\wit{u}_{(c-1,d-1)}+\wit{u}_{(c-2,d-2)}+\cdots+\wit{u}_{(c-d,0)}+X',&\text{ if }c\geq d;\\
 \wit{u}_{(c-1,d-1)}+\wit{u}_{(c-2,d-2)}+\cdots+\wit{u}_{(0,d-c)}+Y',&\text{ if }c<d,
 \end{cases}
\end{equation*}
where $X',Y'\in\sum_{(t_1,t_2)<(c,d)} v^{-1}\bbz[v^{-1}]\wit{u}_{(t_1,t_2)}.$
Hence,
\begin{equation*}
  m^{(c,d)}-m^{(c-1,d-1)}=\wit{u}_{(c,d)}+Z, \text{ where }Z\in\sum_{(t_1,t_2)<(c,d)}v^{-1}\bbz[v^{-1}]\wit{u}_{(t_1,t_2)}.
\end{equation*}
This proves that $\scc_A= m^{(c,d)}-m^{(c-1,d-1)}$ is the canonical basis element associated to $A$ in this case.

Next we prove (2). Fix $a>b$ and let
$$\aligned
M(c,d)&=\sum_{k_1=0}^c(-1)^{c-k_1}\begin{bmatrix}
  a-b-1+c-k_1\\a-b-1
\end{bmatrix}m^{(k_1,d)}-\sum_{l_1=0}^{c-1}(-1)^{c-1-l_1}\begin{bmatrix}
  a-b-2+c-l_1\\a-b-1
\end{bmatrix}m^{(l_1,d-1)}\\
&= \wit{u}_{(c,d)}+\sum_{t_1=0}^{c-1}f^{(c,d)}_{(t_1,d)}\wit{u}_{(t_1,d)}+\sum_{t_2=0}^{d-1}f^{(c,d)}_{(c,t_2)}\wit{u}_{(c,t_2)}+\sum_{(t_1,t_2)\ll(c,d)}(f^{(c,d)}_{(t_1,t_2)}-f^{(c-1,d-1)}_{(t_1,t_2)})\wit{u}_{(t_1,t_2)},
\endaligned$$
where $(t_1,t_2)\ll(c,d)$ means $t_1<c$ and $t_2<d$, and
$$\aligned
\sum_{k_1=0}^c(-1)^{c-k_1}\begin{bmatrix}
  a-b-1+c-k_1\\a-b-1\end{bmatrix}m^{(k_1,d)}&=\sum_{(t_1,t_2)\leqs(c,d)}f^{(c,d)}_{(t_1,t_2)}\wit{u}_{(t_1,t_2)},\text{ and}\\
  \sum_{l_1=0}^{c-1}(-1)^{c-1-l_1}\begin{bmatrix}
  a-b-2+c-l_1\\a-b-1
\end{bmatrix}m^{(l_1,d-1)}&=\sum_{(t_1,t_2)\ll(c,d)}f^{(c-1,d-1)}_{(t_1,t_2)}\wit{u}_{(t_1,t_2)}.\endaligned
 $$
Expending the left hand sides by Lemma~\ref{mk1k2} yields, for $(t_1,t_2)\leqs(c,d)$,
\begin{equation*}
\begin{split}
 &f^{(c,d)}_{(t_1,t_2)}=\sum_{k_1=t_1}^c(-1)^{c-k_1}v^{(a-b-k_1+d+t_1-t_2)(k_1-d-t_1+t_2)}
    [\begin{smallmatrix}a-b-1+c-k_1\\a-b-1\end{smallmatrix}]\ol{[\![\begin{smallmatrix} a+c+d-t_1-t_2\\ k_1-t_1\end{smallmatrix}]\!]}
    \ol{[\![\begin{smallmatrix} b+c+d-t_1-t_2\\ d-t_2\end{smallmatrix}]\!]},\\
 &f^{(c-1,d-1)}_{(t_1,t_2)}=\sum_{l_1=t_1}^{c-1}(-1)^{c-1-l_1}v^{(a-b-l_1+d-1+t_1-t_2)(l_1-d+1-t_1+t_2)}[\begin{smallmatrix}a-b-2+c-l_1\\a-b-1\end{smallmatrix}]
    \ol{[\![\begin{smallmatrix} a+c+d-t_1-t_2\\ l_1-t_1\end{smallmatrix}]\!]}\ol{[\![\begin{smallmatrix} b+c+d-t_1-t_2\\ d-1-t_2\end{smallmatrix}]\!]}.
\end{split}
\end{equation*}
In particular, since $a>b$,
$$f^{(c,d)}_{(c,t_2)}=v^{(a-b+d-t_2)(-d+t_2)}\ol{\left[\!\!\left[\begin{matrix} b+d-t_2\\ d-t_2\end{matrix}\right]\!\!\right]}=v^{-t_2'(a-b+t_2')}\ol{\left[\!\!\left[\begin{matrix} b+t_2'\\ t_2'\end{matrix}\right]\!\!\right]}\in v^{-1}\bbz[v^{-1}]\;\;(t_2'=d-t_2\geq0).$$
and, by Lemma~\ref{technic-lem-of-gauss-poly}(1), we have as seen at the end of the proof of Theorem \ref{theorem for case (2,1)},
$$\aligned
f^{(c,d)}_{(t_1,d)}&=\sum_{k_1=t_1}^c(-1)^{c-k_1}v^{(a-b-k_1+t_1)(k_1-t_1)}\begin{bmatrix}a-b-1+c-k_1\\a-b-1\end{bmatrix}
       \ol{\left[\!\!\left[\begin{matrix} a+c-t_1\\ k_1-t_1\end{matrix}\right]\!\!\right]}\\
&=v^{-t_1'(a+t_1')}\begin{bmatrix}b+t_1'\\t_1'\end{bmatrix}=v^{-t_1'(a-b+t_1')}\ol{\left[\!\!\left[\begin{matrix}b+t_1'\\t_1'\end{matrix}\right]\!\!\right]}\in v^{-1}\bbz[v^{-1}]\;\; (t_1'=c-t_1).
\endaligned$$

Assume now $(t_1,t_2)\ll(c,d)$ and let
$$g^{(c,d)}_{(t_1,t_2)}:=f^{(c,d)}_{(t_1,t_2)}-f^{(c-1,d-1)}_{(t_1,t_2)}.$$
If $(t_1,t_2)=(0,0)$, then
$g_{(0,0)}^{(c,d)} \in v^{-1}\bbz[v^{-1}]$. This is done in Lemma \ref{coeff of (0,0) for a>b} of the Appendix. 


It remains to prove that $g^{(c,d)}_{(t_1,t_2)}\in v^{-1}\bbz[v^{-1}]$ for all $(0,0)<(t_1,t_2)\ll(c,d)$. This follows from the following recursive formula: for all $(0,0)<(t_1,t_2)\leq (c',d')\ll(c,d)$,
$$
g^{(c'+1,d'+1)}_{(t_1,t_2)} =\begin{cases} g^{(c'+1,d')}_{(t_1,t_2-1)} ,&\quad\text{ if }t_2\geq1;\\
g^{(c',d'+1)}_{(t_1-1,0)} ,&\quad\text{ if }t_2=0,\end{cases}$$
which can be seen as follows. 

First, the coefficient $g^{(c'+1,d'+1)}_{(t_1,t_2)}$ of $\wit{u}_{(t_1,t_2)}$ in $M(c'+1,d'+1)$ has the form
\begin{equation*}
\begin{split}
&\sum_{k_1=t_1}^{c'+1}(-1)^{c'+1-k_1}v^{(a-b-k_1+d'+1+t_1-t_2)(k_1-d'-1-t_1+t_2)}[\begin{smallmatrix}
  a-b+c'-k_1\\a-b-1
\end{smallmatrix}]\ol{[\![\begin{smallmatrix} a+c'+d'+2-t_1-t_2\\ k_1-t_1\end{smallmatrix}]\!]}\ol{[\![\begin{smallmatrix} b+c'+d'+2-t_1-t_2\\ d'+1-t_2\end{smallmatrix}]\!]}\\
&\quad-\sum_{l_1=t_1}^{c'}(-1)^{c'-l_1}v^{(a-b-l_1+d'+t_1-t_2)(l_1-d'-t_1+t_2)}[\begin{smallmatrix}
  a-b-1+c'-l_1\\a-b-1
\end{smallmatrix}]\ol{[\![\begin{smallmatrix} a+c'+d'+2-t_1-t_2\\ l_1-t_1\end{smallmatrix}]\!]}\ol{[\![\begin{smallmatrix} b+c'+d'+2-t_1-t_2\\ d'-t_2\end{smallmatrix}]\!]}.
\end{split}
\end{equation*}

If $t_2\geqs 1$, then the coefficient $g^{(c'+1,d')}_{(t_1,t_2-1)}$ of $\wit{u}_{(t_1,t_2-1)}$ in $M(c'+1,d')$ has the form
\begin{equation*}
\begin{split}
  &\sum_{k_1=t_1}^{c'+1}(-1)^{c'+1-k_1}v^{(a-b-k_1+d'+t_1-t_2+1)(k_1-d'-t_1+t_2-1)}[\begin{smallmatrix}
  a-b+c'-k_1\\a-b-1
\end{smallmatrix}]\ol{[\![\begin{smallmatrix} a+c'+1+d'-t_1-t_2+1\\ k_1-t_1\end{smallmatrix}]\!]}\ol{[\![\begin{smallmatrix} b+c'+1+d'-t_1-t_2+1\\ d'-t_2+1\end{smallmatrix}]\!]}\\
&-\!\!\sum_{l_1=t_1}^{c'}(-1)^{c'-l_1}\!v^{(a-b-l_1+d'-1+t_1-t_2+1)(l_1-d'+1-t_1+t_2-1)}[\begin{smallmatrix}
  a-b-1+c'-l_1\\a-b-1
\end{smallmatrix}]\ol{[\![\begin{smallmatrix} a+c'+1+d'-t_1-t_2+1\\ l_1-t_1\end{smallmatrix}]\!]}\ol{[\![\begin{smallmatrix} b+c'+1+d'-t_1-t_2+1\\ d'-1-t_2+1\end{smallmatrix}]\!]},
\end{split}
\end{equation*}
which is the same as that of $\wit{u}_{(t_1,t_2)}$ in $M(c'+1,d'+1)$, proving the first recursive formula.

If $t_2=0,t_1\geqs1$, the coefficient $g^{(c',d'+1)}_{(t_1-1,0)}$ of $\wit{u}_{(t_1-1,0)}$ in $M(c',d'+1)$ has the form
\begin{equation*}
  \begin{split}
  &\sum_{k_1=t_1-1}^{c'}(-1)^{c'-k_1}v^{(a-b-k_1+d'+1+t_1-1)(k_1-d'-1-t_1+1)}[\begin{smallmatrix}
  a-b-1+c'-k_1\\a-b-1
\end{smallmatrix}]\ol{[\![\begin{smallmatrix} a+c'+d'+1-t_1+1\\ k_1-t_1+1\end{smallmatrix}]\!]}\ol{[\![\begin{smallmatrix} b+c'+d'+1-t_1+1\\ d'+1\end{smallmatrix}]\!]}\\
&\quad\;-\sum_{l_1=t_1-1}^{c'-1}(-1)^{c'-1-l_1}v^{(a-b-l_1+d'+t_1-1)(l_1-d'-t_1+1)}[\begin{smallmatrix}
  a-b-2+c'-l_1\\a-b-1
\end{smallmatrix}]\ol{[\![\begin{smallmatrix} a+c'+d'+1-t_1+1\\ l_1-t_1+1\end{smallmatrix}]\!]}\ol{[\![\begin{smallmatrix} b+c'+d'+1-t_1+1\\ d\end{smallmatrix}]\!]}.
  \end{split}
\end{equation*}
Putting $k_1'=k_1+1,l_1'=l_1+1$, we obtain
\begin{equation*}
  \begin{split}
    g^{(c',d'+1)}_{(t_1-1,0)}&=\sum_{k'_1=t_1}^{c'+1}(-1)^{c'-k'_1+1}v^{(a-b-k'_1+d'+1+t_1)(k'_1-d'-1-t_1)}[\begin{smallmatrix}
  a-b+c'-k'_1\\a-b-1
\end{smallmatrix}]\ol{[\![\begin{smallmatrix} a+c'+d'+2-t_1\\ k'_1-t_1\end{smallmatrix}]\!]}\ol{[\![\begin{smallmatrix} b+c'+d'+2-t_1\\ d'+1\end{smallmatrix}]\!]}\\
&-\sum_{l'_1=t_1}^{c'}(-1)^{c'-l_1'}v^{(a-b-l'_1+d'+t_1)(l'_1-d'-t_1)}[\begin{smallmatrix}
  a-b-1+c-l'_1\\a-b-1
\end{smallmatrix}]\ol{[\![\begin{smallmatrix} a+c'+d'+2-t_1\\ l'_1-t_1\end{smallmatrix}]\!]}\ol{[\![\begin{smallmatrix} b+c'+d'+2-t_1\\ d'\end{smallmatrix}]\!]},
  \end{split}
\end{equation*}
which is the same as that of $\wit{u}_{(t_1,0)}$ in $M(c'+1,d'+1)$, proving the second recursive formula.

Repeatedly applying the recursive formula yields, for all $(0,0)<(t_1,t_2)\ll(c,d)$,
$$g^{(c,d)}_{(t_1,t_2)}=g_{(0,0)}^{(c-t_1,d-t_2)}.$$
By Lemma \ref{coeff of (0,0) for a>b} again, $g_{(t_1,t_2)}^{(c,d)}\in v^{-1}\bbz[v^{-1}]$.
This completes the proof of (2).

The proof of (3) can also be reduced by induction to prove that the coefficient of $\wit{u}_{(0,0)}$ belongs to $v^{-1}\bbz[v^{-1}]$, which is given in Lemma \ref{coeff of (0,0) for a>b} of the Appendix.
\end{proof}

\begin{appendix}
\section{The coefficient of $\wit{u}_{(0,0)}$} 
To complete the proof of Theorem \ref{slice22}, we need the following result. We first rewrite the identity in Lemma~\ref{technic-lem-of-gauss-poly}(2) as
  \begin{equation}\label{identity2}
    \sum_{i=0}^\dz(-1)^iv^{i(2\delta-2n-i-1)+2\delta(n+k)}\!\!\ol{\left[\!\!\left[\begin{matrix}
    k-1+i\\k-1
  \end{matrix}\right]\!\!\right]}\ol{\!\!\left[\!\!\left[\begin{matrix}
    m+n\\ \dz-i
  \end{matrix}\right[\!\!\right]}=\sum_{t=0}^{\text{\rm min}\{\dz,n\}}v^{2t(\delta+n+k-t)}\!\!\ol{\left[\!\!\left[\begin{matrix}
    m-k\\\dz-t
  \end{matrix}\right]\!\!\right]}\!\!\ol{\left[\!\!\left[\begin{matrix}
    n\\t
  \end{matrix}\right]\!\!\right]}.
  \end{equation}
for all $m\geqs k\geqs 0,\dz,n\in\bbn$.
\begin{lem}\label{coeff of (0,0) for a>b}
For the numbers $a,b,c,d\in\bbn$ with $c,d\geqs 1$ as given in Theorem \ref{slice22}, we have
$$g^{(c,d)}_{(0,0)}\in v^{-1}\bbz[v^{-1}],$$
where, for $a>b$, 
\begin{equation*}
\begin{split}
g^{(c,d)}_{(0,0)}&= \sum_{k_1=0}^c(-1)^{c-k_1}v^{(a-b-k_1+d)(k_1-d)}\begin{bmatrix}
  a-b-1+c-k_1\\a-b-1
\end{bmatrix}\ol{\left[\!\!\left[\begin{matrix} a+c+d\\ k_1\end{matrix}\right]\!\!\right]}\ol{\left[\!\!\left[\begin{matrix} b+c+d\\ d\end{matrix}\right]\!\!\right]}\\
&\quad\;-\sum_{l_1=0}^{c-1}(-1)^{c-1-l_1}v^{(a-b-l_1+d-1)(l_1-d+1)}\begin{bmatrix}
  a-b-2+c-l_1\\a-b-1
\end{bmatrix}\ol{\left[\!\!\left[\begin{matrix} a+c+d\\ l_1\end{matrix}\right]\!\!\right]}\ol{\left[\!\!\left[\begin{matrix} b+c+d\\ d-1\end{matrix}\right]\!\!\right]};
\end{split}
\end{equation*}
while, for $a<b$, 
\begin{equation*}
\begin{split}
g^{(c,d)}_{(0,0)}&= \sum_{k_1=0}^d(-1)^{d-k_1}v^{(b-a-k_1+c)(k_1-c)}\begin{bmatrix}
  b-a-1+d-k_1\\b-a-1
\end{bmatrix}\ol{\left[\!\!\left[\begin{matrix} a+c+d\\ c\end{matrix}\right]\!\!\right]}\ol{\left[\!\!\left[\begin{matrix} b+c+d\\ k_1\end{matrix}\right]\!\!\right]}\\
&\quad\;-\sum_{l_1=0}^{d-1}(-1)^{d-1-l_1}v^{(b-a-l_1+c-1)(l_1-c+1)}\begin{bmatrix}
  b-a-2+d-l_1\\b-a-1
\end{bmatrix}\ol{\left[\!\!\left[\begin{matrix} a+c+d\\ c-1\end{matrix}\right]\!\!\right]}\ol{\left[\!\!\left[\begin{matrix} b+c+d\\ l_1\end{matrix}\right]\!\!\right]}.
\end{split}
\end{equation*}

\end{lem}
\begin{proof}
We only prove the $a>b$ case, the other case can be proved similarly. Rewrite $g_{(0,0)}^{(c,d)}$ as
\begin{equation*}
\begin{split}
 g_{(0,0)}^{(c,d)}=&\sum_{k_1=0}^c(-1)^{c-k_1}v^{(a-b-k_1+d)(k_1-d)+(c-k_1)(a-b-1)}\ol{[\![\begin{smallmatrix}
  a-b-1+c-k_1\\a-b-1
\end{smallmatrix}]\!]}\ol{[\![\begin{smallmatrix} a+c+d\\ k_1\end{smallmatrix}]\!]}\cdot\ol{[\![\begin{smallmatrix} b+c+d\\ d\end{smallmatrix}]\!]}\\
&-\sum_{l_1=0}^{c-1}(-1)^{c-1-l_1}v^{(a-b-l_1+d-1)(l_1-d+1)+(a-b-1)(c-1-l_1)}\ol{[\![\begin{smallmatrix}
  a-b-2+c-l_1\\a-b-1
\end{smallmatrix}]\!]}\ol{[\![\begin{smallmatrix}a+c+d\\ l_1\end{smallmatrix}]\!]}\cdot\ol{[\![\begin{smallmatrix}b+c+d\\ d-1\end{smallmatrix}]\!]}.\\
\end{split}
\end{equation*}

If $c\leqs d$, then rearranging gives
\begin{equation*}
\begin{split}
 g^{(c,d)}_{(0,0)}&=(-1)^{c}v^{-d(a-b+d)+c(a-b-1)}\ol{[\![\begin{smallmatrix}
  a-b-1+c\\a-b-1
\end{smallmatrix}]\!]}\ol{[\![\begin{smallmatrix} b+c+d\\ d\end{smallmatrix}]\!]}\\
+&\sum_{k_1=1}^c(-1)^{c-k_1}v^{(a-b-k_1+d)(k_1-d)+(c-k_1)(a-b-1)}\ol{[\![\begin{smallmatrix}
  a-b-1+c-k_1\\a-b-1
\end{smallmatrix}]\!]}\big(\ol{[\![\begin{smallmatrix} a+c+d\\ k_1\end{smallmatrix}]\!]}\ol{[\![\begin{smallmatrix} b+c+d\\ d\end{smallmatrix}]\!]}-\ol{[\![\begin{smallmatrix} a+c+d\\ k_1-1\end{smallmatrix}]\!]}\ol{[\![\begin{smallmatrix} b+c+d\\ d-1\end{smallmatrix}]\!]}\big).\\
\end{split}
\end{equation*}
Since $a>b$ and $c\leqs d$,  $-d(a-b+d)+c(a-b-1)=(a-b-1)(c-d)-d(1+d)<0$ and so the first term is in $v^{-1}\bbz[v^{-1}]$. Since the difference of the product of Gaussian polynomials is in $v^{-1}\bbz[v^{-1}]$, and 
$(a-b-k_1+d)(k_1-d)+(c-k_1)(a-b-1)=(a-b-1)(c-d)+(1+d-k_1)(k_1-d)\leqs 0$, this proves $g^{(c,d)}_{(0,0)}\in v^{-1}\bbz[v^{-1}]$ in this case.

We now assume $c>d$. By rearranging the exponents of $v$, $g^{(c,d)}_{(0,0)}$ has the form
\begin{equation*}
\begin{split}
 g_{(0,0)}^{(c,d)}
=&v^{-(a-b)(c+d)-c^2-d^2}\ol{[\![\begin{smallmatrix} b+c+d\\ d\end{smallmatrix}]\!]}\cdot S_1-v^{2(a-b+c+d-1)-(a-b)(c+d)-c^2-d^2}\ol{[\![\begin{smallmatrix}b+c+d\\ d-1\end{smallmatrix}]\!]}\cdot S_2\\
\end{split}
\end{equation*}

where
\begin{equation*}
  \begin{split}
    S_1&=\sum_{k_1=0}^c(-1)^{c-k_1}v^{(c-k_1)(c+k_1-2d-1)+2c(a-b+d)}\ol{[\![\begin{smallmatrix}
  a-b-1+c-k_1\\a-b-1
\end{smallmatrix}]\!]}\ol{[\![\begin{smallmatrix} a+c+d\\ k_1\end{smallmatrix}]\!]}\\
    S_2&=\sum_{l_1=0}^{c-1}(-1)^{c-1-l_1}v^{(c-1-l_1)(c+l_1-2d)+2(c-1)(a-b+d-1)}\ol{[\![\begin{smallmatrix}
  a-b-2+c-l_1\\a-b-1
\end{smallmatrix}]\!]}\ol{[\![\begin{smallmatrix} a+c+d\\ l_1\end{smallmatrix}]\!]}
  \end{split}
\end{equation*}

Applying \eqref{identity2} (i.e., Lemma \ref{technic-lem-of-gauss-poly}(2)) to $S_1$ with $k=a-b,m=a+c,n=d,i=c-k_1,\dz=c$ and to $S_2$
with $k=a-b,m=a+c+1,n=d-1,i=c-1-l_1,\dz=c-1$ yields

\begin{equation*}
  \begin{split}
    S_1&=\sum_{t=0}^dv^{2t(a+c+d-b-t)}\ol{[\![\begin{smallmatrix}
  b+c\\c-t
\end{smallmatrix}]\!]}\ol{[\![\begin{smallmatrix} d\\ t\end{smallmatrix}]\!]},\quad
    S_2=\sum_{t=0}^{d-1}v^{2t(a+c+d-b-2-t)}\ol{[\![\begin{smallmatrix}
  b+c+1\\c-1-t
\end{smallmatrix}]\!]}\ol{[\![\begin{smallmatrix} d-1\\t\end{smallmatrix}]\!]}
  \end{split}
\end{equation*}
Thus,

\begin{equation*}
\begin{split}
 g_{(0,0)}^{(c,d)}
&=v^{-(a-b)(c+d)-c^2-d^2}\ol{[\![\begin{smallmatrix} b+c+d\\ d\end{smallmatrix}]\!]}\ol{[\![\begin{smallmatrix}
  b+c\\c
\end{smallmatrix}]\!]}\\
&\quad+v^{-(a-b)(c+d)-c^2-d^2}\ol{[\![\begin{smallmatrix} b+c+d\\ d\end{smallmatrix}]\!]}\big(\sum_{t=1}^dv^{2t(a+c+d-b-t)}\ol{[\![\begin{smallmatrix}
  b+c\\c-t
\end{smallmatrix}]\!]}\ol{[\![\begin{smallmatrix} d\\ t\end{smallmatrix}]\!]}\big)\\
&\quad-v^{2(a-b+c+d-1)-(a-b)(c+d)-c^2-d^2}\ol{[\![\begin{smallmatrix}b+c+d\\ d-1\end{smallmatrix}]\!]}\big(\sum_{t=0}^{d-1}v^{2t(a+c+d-b-2-t)}\ol{[\![\begin{smallmatrix}
  b+c+1\\c-1-t
\end{smallmatrix}]\!]}\ol{[\![\begin{smallmatrix} d-1\\t\end{smallmatrix}]\!]}\big)\\
\end{split}
\end{equation*}
Changing the running index $t\in\{0,1,\ldots,d-1\}$ to $t'=t+1\in\{1,2,\ldots,d\}$ in the last sum gives

\begin{equation*}
  \begin{split}
    g_{(0,0)}^{(c,d)}=&v^{-(a-b)(c+d)-c^2-d^2}\ol{[\![\begin{smallmatrix} b+c+d\\ d\end{smallmatrix}]\!]}\ol{[\![\begin{smallmatrix} b+c\\ c\end{smallmatrix}]\!]}\\
    &+\sum_{t=1}^dv^{-(a-b)(c+d)-c^2-d^2+2t(a+c+d-b-t)}\big(\ol{[\![\begin{smallmatrix} b+c+d\\ d\end{smallmatrix}]\!]}\ol{[\![\begin{smallmatrix} b+c\\ c-t\end{smallmatrix}]\!]}\ol{[\![\begin{smallmatrix} d\\ t\end{smallmatrix}]\!]}-\ol{[\![\begin{smallmatrix} b+c+d\\ d-1\end{smallmatrix}]\!]}\ol{[\![\begin{smallmatrix} b+c+1\\ c-t\end{smallmatrix}]\!]}\ol{[\![\begin{smallmatrix} d-1\\ t-1\end{smallmatrix}]\!]}\big).
  \end{split}
\end{equation*}
The first term is clear in $v^{-1}\bbz[v^{-1}]$ since $a>b$. Now, $c>d$ implies that 

\begin{equation*}
\aligned
&\quad\,  -(a-b)(c+d)-c^2-d^2+2t(a+c+d-b-t) \\
&\leqs  -(a-b)(c+d)-c^2-d^2+2d(a+c-b)\\
  &=-(c-d)(a-b+c-d)<0
\endaligned
\end{equation*}
for any $t=1,2,\cdots,d$. Hence, $g_{(0,0)}^{(c,d)}\in v^{-1}\bbz[v^{-1}]$.
\end{proof}
\end{appendix}




\begin{thebibliography}{10}
\bibitem{BeilinsonLusztigMacPherson1990geometric}
A.~A. Beilinson, G.~Lusztig and R.~MacPherson, \emph{A geometric setting for
  the quantum deformation of ${GL}_n$}, Duke Math. J. \textbf{61} (1990),
  no.~2, 655--677.

\bibitem{Bongartz1996degenerations}
K.~Bongartz, \emph{On degenerations and extensions of finite dimensional
  modules}, Adv. Math. \textbf{121} (1996), 245--287.

\bibitem{DengDu2005monomial}
B.~Deng and J.~Du, \emph{Monomial bases for quantum affine $\mathfrak{sl}_n$},
  Adv. Math. \textbf{191} (2005), 276--304.

\bibitem{DengDuFu2012double}
B.~Deng, J.~Du and Q.~Fu, \emph{{A double Hall algebra approach to affine
  quantum Schur-Weyl theory}}, no. 401, Cambridge University Press, 2012.

\bibitem{DengDuParashallWang2008finite}
B.~Deng, J.~Du, B.~Parashall and J.~Wang, \emph{Finite dimensional algebras and
  quantum groups}, Mathematical Surveys and Monographs, vol. 150, American
  Mathematical Society, Providence, RI, 2008.

\bibitem{DengDuXiao2007generic}
B.~Deng, J.~Du and J.~Xiao, \emph{Generic extensions and canonical bases for
  cyclic quivers}, Canad. J. Math. \textbf{59} (2007), no.~6, 1260--1283.

\bibitem{Drinfeld1988new}
V.~Drinfeld, \emph{A new realization of {Y}angians and quantized affine
  algebras}, Sov. Math. Dokl. \textbf{36} (1988), no.~2, 212--216.

\bibitem{Du1994ic}
J.~Du, \emph{{IC} bases and quantum linear groups}, Proc. Sympos. Pure Math.
  \textbf{56} (1994), 135--148.

\bibitem{DuFu2010modified}
J.~Du and Q.~Fu, \emph{A modified {BLM} approach to quantum affine
  $\fkg\fkl_n$}, Math. Z. \textbf{266} (2010), no.~4, 747--781.

\bibitem{DuFu2015quantum}
J.~Du and Q.~Fu, \emph{Quantum affine ${\fkg\fkl}_n$ via {H}ecke algebras},
  Adv. Math. \textbf{282} (2015), 23--46.

\bibitem{DF16}
J.~Du and Q.~Fu, \emph{The Integral Quantum loop algebra of $\mathfrak{gl}_n$}, preprint.

\bibitem{Guo1995hallpoly}
J.~Y. Guo, \emph{The {H}all polynomials of a cyclic serial algebra}, Comm.
  Algebra. \textbf{23} (1995), 743--751.

\bibitem{FLLLW} Z. Fan, C. Lai, Y. Li, L. Luo, W. Wang. {\em Affine flag varieties and quantum symmetric pairs}, arXiv:1602.04383 (108pp)

\bibitem{FL}Z. Fan and Y. Li, {\it Positivity of canonical bases under comultiplication}, arXiv:1511.02434v3.

\bibitem{GV}
V. Ginzburg and E. Vasserot, {\em Langlands reciprocity for affine quantum groups of type $A_n$}, Internat. Math. Res. Notices 1993, 67--85.

\bibitem{Hubery2005symmetric}
A.~Hubery, \emph{Symmetric functions and the center of the {R}ingel--{H}all
  algebra of a cyclic quiver}, Math. Z. \textbf{251} (2005), no.~3, 705--719.

\bibitem{Jantzen1995lectures}
J.~C. Jantzen, \emph{Lectures on quantum groups}, Graduate Studies in
  Mathematics, vol.~6, American Mathematical Society, Providence, RI, 1995.


\bibitem{Lusztig1990canonical}
G.~Lusztig, \emph{Canonical bases arising from quantized enveloping algebras},
  J. Amer. Math. Soc. \textbf{3} (1990), no.~2, 447--498.

\bibitem{Lusztig1992affine}
G.~Lusztig, \emph{Affine quivers and canonical bases}, Inst. Hautes \'etudes
  Sci. Publ. Math. (1992), no.~76, 111--163.

\bibitem{Lusztig1993introduction}
G.~Lusztig, \emph{Introduction to quantum groups}, Progress in Math., vol. 110,
  Birkh\"auser, Boston, 1993.

\bibitem{Lu93a} G. Lusztig, \emph{Tight monomials in quantized enveloping algebras}, in: Quantum Deformations of Algebras and Their Representations,
 Israel Math. Conf. Proc., vol. 7, 1993, pp. 117--132.
\bibitem{Lu99}
G. Lusztig, {\em Aperiodicity in quantum affine $\frak{gl}_n$},
Asian J. Math. {\bf 3} (1999),  147--177.

\bibitem{Reineke2001generic}
M.~Reineke, \emph{Generic extensions and multiplicative bases of quantum groups
  at $q=0$}, Represent. Theory. \textbf{5} (2001), 147--163.

\bibitem{RInv}C. M. Ringel, \emph{Hall algebras and quantum groups}, Invent. Math. \textbf{101}(1990), 583--592.

\bibitem{RRev}C. M. Ringel, \emph{Hall algebras revisited}, in Quantum Deformations of Algebras and
Their Representations, A. Joseph, S. Shnider (eds.), Israel Mathematical Conference
Proceedings, no. 7, Bar-Ilan University, Bar-Ilan, 1993, pp. 171--176.

\bibitem{Ringel1993composition}
C.~M. Ringel, \emph{The composition algebra of a cyclic quiver}, Proc. London
  Math. Soc. \textbf{66} (1993), 507--537.

\bibitem{Schiffmann2000hall}
O.~Schiffmann, \emph{The {H}all algebra of a cyclic quiver and canonical bases
  of {F}ock spaces}, Int. Math. Res. Not. \textbf{2000} (2000), no.~8,
  413--440.

  \bibitem{Sc06}
O.~Schiffmann, \emph{Lectures on Hall algebras}, arXiv:math/0611617.

\bibitem{VaragnoloVasserot1999decomposition}
M.~Varagnolo and E.~Vasserot, \emph{On the decomposition matrices of the
  quantized {S}chur algebra}, Duke Math. J. \textbf{100} (1999), 267--297.

\bibitem{XicanonicalA31999}
N.~Xi, \emph{Canonical basis for type {$A_3$}}, Comm. Algebra. \textbf{27}
  (1999), no.~11, 5703--5710.
 
\bibitem{XicanonicalB21999}
N.~Xi, \emph{Canonical basis for type {$B_2$}}, J. Algebra. \textbf{214}
  (1999), no.~1, 8--21.

\bibitem{Xiao1997drinfeld}
J.~Xiao, \emph{Drinfeld double and {R}ingel--{G}reen theory of {H}all
  algebras}, J. Algebra. \textbf{190} (1997), no.~1, 100--144.

\bibitem{Zwara1997degenerations}
G.~Zwara, \emph{Degenerations for modules over representation-finite biserial
  algebras}, J. Algebra. \textbf{198} (1997), no.~2, 563--581.

\end{thebibliography}
\end{document}